\documentclass[11pt]{amsart}
\usepackage {amsmath, amssymb,   epsfig, a4wide, enumerate, psfrag}
\usepackage[latin1]{inputenc}
\usepackage[english]{babel}
\date{\today}

\author{Romain Dujardin}
\thanks{Research   partially supported by ANR project BERKO}

\address{CMLS \\ \'Ecole Polytechnique \\ 91128 Palaiseau\\
         France}
\email{dujardin@math.polytechnique.fr}

\title{Bifurcation currents and equidistribution in parameter space}

\dedicatory{To John Milnor for his  80$\;^{th}$  birthday} 

\newcommand{\cc}{\mathbb{C}}

\newcommand{\re}{\mathbb{R}}
\newcommand{\dd}{\mathbb{D}}
\newcommand{\zz}{\mathbb{Z}}
\newcommand{\nn}{\mathbb{N}}
\newcommand{\pp}{\mathbb{P}}
\newcommand{\e}{\varepsilon}
\newcommand{\cv}{\rightarrow}

\newcommand{\fr}{\partial}
\newcommand{\om}{\Omega}
\newcommand{\set}[1]{\left\{#1\right\}}
\newcommand{\norm}[1]{\left\Vert#1\right\Vert}
\newcommand{\abs}[1]{\left\vert#1\right\vert}
\newcommand{\cd}{{\cc^2}}

\newcommand{\pu}{{\mathbb{P}^1}}
\newcommand{\pk}{{\mathbb{P}^k}}

\newcommand{\rest}[1]{ \arrowvert_{#1}}
\newcommand{\m}{{\bf M}}
\newcommand{\unsur}[1]{\frac{1}{#1}}

\newcommand{\lrpar}[1]{\left(#1\right)}
\newcommand{\bra}[1]{\left\langle #1\right\rangle}

\newcommand{\la}{\lambda}
\newcommand{\La}{\Lambda}
\newcommand{\SL}{\mathrm{SL}(2,\mathbb C)}
\newcommand{\PSL}{\mathrm{PSL}(2,\mathbb C)}
\newcommand{\tbif}{T_{\mathrm{bif}}}
\newcommand{\mubif}{\mu_{\mathrm{bif}}}

\newcommand{\inv}{^{-1}}
\newcommand{\hL}{\widehat{\Lambda}}
\newcommand{\hT}{\widehat{T}}
\newcommand{\hf}{\widehat{f}}
\renewcommand{\hom}{\widehat\omega}
\newcommand{\per}{\mathrm{Per}}

\newcommand{\percrit}{\mathrm{PerCrit}}
\newcommand{\prepercrit}{\mathrm{PreperCrit}}
\newcommand{\C}{\mathcal{C}}

\newcommand{\itm}{\item[-]}

\DeclareMathOperator{\supp}{Supp}
\DeclareMathOperator{\vol}{vol}
\DeclareMathOperator{\Int}{Int}

\DeclareMathOperator{\tr}{tr}

\DeclareMathOperator{\length}{length}

\DeclareMathOperator{\bif}{Bif}
\DeclareMathOperator{\stab}{Stab}

\newcommand{\diamant}{\medskip \begin{center}$\diamond$\end{center}\medskip}

\newtheorem{prop}{Proposition} [section]
\newtheorem{thm}[prop] {Theorem}
\newtheorem{defi}[prop] {Definition}
\newtheorem{lem}[prop] {Lemma}
\newtheorem{cor}[prop]{Corollary}

\newtheorem{propdef}[prop]{Proposition-Definition}

\newtheorem{question}[prop]{Question}
\newtheorem{conjecture}[prop]{Conjecture}

\theoremstyle{remark}
\newtheorem{exam}[prop]{Example}
\newtheorem{rmk}[prop]{Remark}

\begin{document}

\maketitle
~
\begin{abstract}
We review the use of techniques of positive currents for the study of parameter spaces of one-dimensional holomorphic dynamical systems (rational mappings on $\mathbb{P}^1$ or subgroups of the Möbius group $\mathrm{PSL}(2, \mathbb{C})$). The topics covered include: the construction of bifurcation currents and the characterization of their supports, the equidistribution properties of dynamically defined subvarieties on parameter space.
\end{abstract}
~
\tableofcontents

\newpage

\section*{Introduction}

 Let $(f_\la)_{\la\in\La}$ be a holomorphic family  of
 dynamical systems acting 
 on the Riemann sphere $\pu$, parameterized by a complex manifold $\La$. The ``dynamical systems" in consideration here can be polynomial or rational mappings on $\pu$, as well as 
  groups of Möbius transformations. It is a very basic idea that the %analysis of the  
product dynamics $\widehat f$ acting on $\La\times \pu$ by 
 $\widehat{f}(\la,z) =(\la, f_\la(z))$ is an important  source of 
 information on the bifurcation theory of the family. 
The input of techniques  from higher dimensional  holomorphic dynamics into this problem
 recently led to a number of interesting new results %in the bifurcation theory of one dimensional dynamical systems,  
 in this area, especially when the parameter space 
 $\La$ is multidimensional.
Our purpose in this paper is to review these recent developments.
   
   \medskip
   
The main new idea that has arisen  from this interaction 
 is the use of  positive closed currents. 
  We will see that the  consideration of the dynamics of $\widehat f$    gives rise to a number of  
  interesting currents on $\La\times \pu$ and $\La$.  Positive currents have an underlying measurable 
  structure, so it would be fair to say that we are studying these
 parameter spaces at a measurable level, 
  somehow in the spirit of the ergodic theoretic approach to dynamics\footnote{In this respect it is instructive to compare this with the more topological point of view of Branner and Hubbard, which was summarized 20 years ago by Branner in \cite{branner-turning}.}.
  
  \medskip
    
  A basic way in general to construct and study dynamical currents is to view them as limits of sequences of dynamically defined subvarieties. 
  This will be another major theme in this paper. 

\medskip

We will try as much as possible to emphasize the similarities  between  methods of 
higher dimensional dynamics, of  the study of families of rational maps and that  of Möbius subgroups. We will also  state a number of open questions, to foster  further developments of this theory.  

\medskip

An interesting outcome of these methods is the possibility of studying ``higher codimensional" 
phenomena --like the property for a rational map of having several periodic critical points. These 
phenomena are difficult to grasp using elementary complex analysis techniques because of the failure 
of Montel's theorem in higher dimension. In the same vein, we will see that when $\dim(\La)>1$, the bifurcation locus of a family of rational maps possesses a hierarchical structure, which may  conveniently be formalized using bifurcation currents. When $(f_\la)_{\la\in \La}$ is the family of polynomials of degree $d\geq 3$, the smaller of these successive bifurcation loci is the right analogue of the Mandelbrot set in higher degree, with whom it shares many important properties.

\diamant

{\bf Contents.} Let us now outline the contents of this article.   Section \ref{sec:prologue} is of general nature. We explain how the non-normality of a sequence of holomorphic mappings $f_n:\La\cv X$ between complex manifolds is related to certain closed positive currents of bidegree (1,1)  on $\La$. We also show that the  preimages under $f_n$ of hypersurfaces of $X$ tend to be equidistributed. 
This  will  provide --at least at the conceptual level--
a uniform framework for many of the subsequent results.

Since these facts are not so easy to extract from the literature,  we explain them in full detail, therefore the presentation is a bit technical. The reader who wants to dive directly into holomorphic dynamics is advised to skip this section on a first reading. 

\medskip

Sections \ref{sec:bifcur} and \ref{sec:higher} are devoted to the study of bifurcation currents for polynomials and rational maps on $\pu$, which is the most developed part of the theory. In Section \ref{sec:bifcur} we present two (related) constructions  of   bifurcation currents  of bidegree (1,1): the ``absolute" bifurcation current $\tbif$ and the bifurcation current associated to a marked critical point.  In both cases, the support of the bifurcation current is equal to the corresponding bifurcation locus.
 We also show that these currents describe the asymptotic distribution of  families of  dynamically defined codimension 1 subsets of parameter space. More precisely we will be interested in the families of  hypersurfaces $\percrit(n,k)$ (resp.  $\per(n,w)$) defined by the condition that a critical point satisfies $f^n(c)=f^k(c)$ (resp. $f$ possesses a periodic $n$-cycle of multiplier $w$).
 
 \medskip
 
In Section \ref{sec:higher}, we study   ``higher" bifurcation currents, which are obtained by taking exterior products of the previous ones. We will develop the idea that the supports of these currents define a dynamically meaningful filtration of the bifurcation locus, and   try to characterize them precisely. We also explain why bifurcation currents should display some laminar structure in parts of parameter space, and give some results in this direction. 

Many of the proofs will be sketched,  the reader being referred to the original papers for complete arguments.  Let us also  mention a recent set of lectures notes  by   Berteloot \cite{berteloot survey} which covers most of this material with   greater detail (and complete proofs). 

\medskip

In Section \ref{sec:kleinbif} we introduce   currents associated to bifurcations 
of holomorphic families of subgroups of $\PSL$, which is in a sense  the counterpart of Section \ref{sec:bifcur} in the Kleinian groups setting.    The existence of such a counterpart is in accordance with the so-called    {\em Sullivan dictionary} between rational and Kleinian group dynamics, nevertheless its practical  implementation requires a   number of new ideas.    To be specific, let $(\rho_\la)$ be a holomorphic family of representations of a given finitely generated group into $\PSL$ (satisfying certain natural assumptions). We construct a bifurcation current on $\La$ associated to a random walk on $G$. As before, % the main  properties of this current are  
%the fact that it
this current  is supported precisely on    the ``bifurcation locus'' of the family, and   it describes the asymptotic distribution of natural codimension 1 subsets of parameter space. 
We will see that the key technical 
 ingredient here is the ergodic theory of  random products of matrices. 
 
  \medskip

In Section \ref{sec:further} we outline  some possible extensions of the theory. An obvious generalization would be to consider rational mappings in higher dimension. A basic  difficulty is 
 that in that setting the    understanding of bifurcation phenomena is  still rather poor.

\medskip

We do not include a general discussion about plurisubharmonic (psh for short) functions and positive currents. Good reference sources for this are the books by   Demailly  \cite[Chap. I and III]{demailly} and Hörmander \cite[Chap. 4]{hormander}. See also \cite{cantat survey} in this volume for a short presentation. We do not require  much knowledge in holomorphic dynamics,   except for the basic properties of the maximal entropy measure \cite{lyubich, flm}.

\diamant

\noindent{\bf Bibliographical overview.} 
Let us briefly review the main references that we will be considering in the paper.
Bifurcation currents were introduced by DeMarco in \cite{demarco1}. In this paper she constructs a   current $\tbif$ on any holomorphic family of rational maps, whose support is the bifurcation locus. This current is defined in terms of the critical points. In \cite{demarco2}, she proves a formula for the Lyapunov exponent of a rational map on $\pu$ (relative to its maximal entropy measure), and deduces from this that $\tbif$ is the $dd^c$ of the Lyapunov exponent function.

In \cite{bas-ber1}, Bassanelli and Berteloot generalize DeMarco's formula to higher dimensional rational maps, and initiate the study of the higher exterior powers of the bifurcation current (associated to rational mappings on $\pu$), 
by showing that   $\supp(\tbif^k)$ is accumulated by parameters possessing $k$ indifferent cycles.
In \cite{preper}, Favre and the author study the asymptotic distribution of the family of hypersurfaces  $\percrit(n,k)$. The structure of  the space of polynomials of degree $d$ is also investigated, with emphasis on the higher dimensional analogue of the boundary of the Mandelbrot set. A finer description is given in the particular  case  of cubic polynomials in \cite{cubic}. 
  
  Several equidistribution theorems for the  family of hypersurfaces
    $\per(n,w)$ are obtained by Bassanelli and Berteloot  in \cite{bas-ber2, bas-ber3}. \cite{bas-ber2} also discusses the laminarity properties of bifurcation currents in the space of quadratic rational maps. 

In  \cite{buff epstein}, Buff and Epstein develop a method based on transversality ideas to characterize the supports of certain ``higher" bifurcation currents. This  was recently  generalized by 
Gauthier   \cite{gauthier}, leading in particular to Hausdorff dimension estimates for the supports of these currents, which generalize Shishikura's famous result that the boundary of the Mandelbrot set has dimension 2. 

Bifurcation currents for families of subgroups of $\PSL$, satisfying properties similar to the above,  were designed by Deroin and the author in \cite{kleinbif}. 

Related developments in higher dimensional holomorphic dynamics can be found in \cite{ds-pl, pham}. 
 
\diamant

We close this introduction with a few words  on the connexions between 
these ideas and the work  of Milnor. 
Alone or with coauthors, he wrote a number of papers, most quite influential, 
 on parameter spaces of polynomials and 
rational functions \cite{ milnor cubic, milnor quadratic, milnor components, 
milnor bicritical, milnor smooth, bkm}. 
Common features among these articles include the  emphasis on multidimensional  issues, and the role played by  subvarieties of parameter space, like the $\percrit(n,k)$ and the $\per(n,w)$.  
I hope he will appreciate the way in which these ideas  
reappear here. 
 
 \medskip

Many thanks to  Serge Cantat, Charles Favre, and Thomas Gauthier for their  useful comments.

\section{Prologue: normal families, currents and equidistribution}\label{sec:prologue}

Let $\La$ be a complex manifold of dimension $d$, and $X$ be  a compact K\"ahler manifold of dimension $k$, endowed    with    K\"ahler form $\omega$. Let $(f_n)$ be a sequence of holomorphic mappings from $\La$ to $X$. In this section we explain a basic construction relating the non normality of the sequence $(f_n)$  and certain positive (1,1) currents on $\La$. When applied to particular situations it will give rise to various bifurcation currents. This construction is also  related to 
 higher dimensional holomorphic dynamics since we   may take $\La = X$ and $f_n$ be the family of iterates of a given self-map on $X$.

\medskip

The problems we consider are local on $\La$ so without loss of generality we assume that $\La$ is an open ball in $\cc^{d}$. We say that 
the family $(f_n)$ is {\em quasi-normal} if for every subsequence of $(f_n)$ (still denoted by $(f_n)$) there exists a  further subsequence $(f_{n_j})$ and   an analytic subvariety $E\subset \La$ such that  $(f_{n_j})$ is a normal family on $\La\setminus E$ (see \cite{ivashkovich} for a discussion of this and other related notions).  

\subsection{A normality criterion} The following result is a variation on well-known ideas, but it is apparently new. 

\begin{thm}\label{thm:volume}
Let $\La$ and $X$ be as above, and $(f_n)$ be a sequence of holomorphic mappings from $\La$ to $X$. 
 If the sequence of bidegree (1,1) currents 
  $f_n^*\omega$ has locally uniformly bounded mass on $\La$, then the family $(f_n)$ is quasi-normal on $\La$. 
\end{thm} 

Recall that the mass of a positive current of bidegree (1,1) in an open set $\om\subset \cc^d$ is defined by $\m_\om(T) = 
\sup \abs{\bra{T,\varphi}}$, where  $\varphi$ ranges among test $(d-1, d-1)$ forms $\sum\varphi_{I,J}dz_I\wedge d\overline{z}_J$ with $\norm{\varphi_{I,J}}_{L^\infty} \leq 1$. 

A few  comments are in order here. %There are several natural ways of defining what it means for a sequence of holomorphic (or meromorphic) mappings to be ``weakly normal'' in some sense. This is subtle matter and it is not the point here to discuss this (see \cite{ivashkovich} for an interesting account)\note{en fait on s'en fout}. Here our focus is on the relationship with (1,1) currents on $\La$. 
First; if $d=1$, the result is a well-known consequence of Bishop's criterion for the normality of a sequence of analytic sets \cite[\S 15.5]{chirka} (see Lemma \ref{lem:dim1}  below for the proof). The point here is that if $d>1$ our assumption does \textit{not} imply that the volumes of the graphs of the $f_n$ are locally bounded (see \cite[Example 5.1]{ivashkovich}). Secondly, it is clear that the converse of Theorem \ref{thm:volume} is false, i.e. there exist quasi-normal families such that $f_n^*\omega$ has unbounded mass. For this, take any sequence of holomorphic mappings   $\dd\cv\pu$, converging on compact subsets of $\dd^*$ to $z\mapsto   \exp(1/z)$.

\begin{lem}\label{lem:dim1}
Theorem \ref{thm:volume} holds when $\dim(\La) =1$.  
\end{lem}
 
\begin{proof}
Let $\Gamma(f_n)\subset \La\times X$ be the graph of $f_n$. Let $\pi_1$, $\pi_2$ be the   first and second projections on $\La\times X$. Let $\beta$ be the standard K\"ahler form on    $\La\subset \cc$. Then if $U\Subset \La$, the volume of $\Gamma(f_n)\cap {\pi_1\inv(U)}$ relative to the product Hermitian structure equals 
  $$\int_{\Gamma(f_n)\cap {\pi_1\inv(U)} }\pi_1^*\beta + \pi_2^*\omega =\vol_\La(U) +   \int_U \lrpar{\pi_2\circ \left(\pi_1\rest{\Gamma(f_n)}\inv\right)}^*\omega = \vol_\La(U) + \int_U  f_n^*\omega.$$ Therefore our assumption implies that the volumes of the analytic sets $\Gamma(f_n)\cap {\pi_1\inv(U)}$ are   uniformly 
bounded\footnote{This is where we use the assumption that $\dim(\La)= 1$. In higher dimension, to estimate this volume one needs to integrate the  
exterior power   $(\pi_1^*\beta + \pi_2^*\omega)^{d}$ where $d=\dim(\La)$.}. By Bishop's theorem one can extract a subsequence $n_j$ such that the $\Gamma(f_{n_j})\cap {\pi_1\inv(U)}$ converge in the Hausdorff topology to a one-dimensional analytic set $\Gamma$ of ${\pi_1\inv(U)}$. 

We claim that $\Gamma$ is the union of a graph and finitely many vertical curves. Here vertical means that it projects to a point on $\La$. Indeed, first note that by Lelong's lower bound on the volume of an analytic set \cite[\S 15.3]{chirka}, the volume of any analytic subset of $X$ is uniformly bounded below. Since $\vol_{\La\times X}(\Gamma)\leq \liminf \vol_{\La\times X}(\Gamma(f_{n_j}))$, this implies that $\Gamma$ contains only finitely many vertical components.  Let 
$E\subset\La$ be the projection of these components.  We claim that the $f_{n_j}$ converge locally uniformly outside $E$. Indeed let $V\subset \La$ be a connected open subset disjoint from $E$. Since the $\Gamma(f_{n_j})$ converge in the Hausdorff topology, we see that $\Gamma\cap \pi_1\inv(V)$ is non empty. Now since $\pi_1$ is proper  we infer that $\pi_1\rest{\Gamma\cap \pi_1\inv(V)}$ is a branched covering, which must be of degree 1 (for if not, generic fibers of $\pi_1$ would intersect $\Gamma(f_{n_j})$ in several points for large $j$). We conclude that $\Gamma$ is a graph over   $V$, of   some $f:V\cv X$, and that the $f_{n_j}$ converge uniformly to $f$ there. 
\end{proof}

The possibility of vertical components of $\Gamma$ over a locally finite set is known as the {\em bubbling phenomenon}. An   important consequence of the proof is that there exists a constant $\delta_0$ (any number smaller than the infimum of the  volumes of 1-dimensional subvarieties of $X$ will do) such that if $\limsup \int_V f_n^*\omega \leq \delta_0$, then no bubbling occurs in $V$. An easy compactness argument yields the following:

\begin{cor}\label{cor:derivative}
 Let $V$ be a one dimensional disk and let $\delta_0>0$ be as above. 
For every $K\Subset V$, there exists a constant $C(V,K,\delta_0)$   so that 
if $f:V\cv X$ is a holomorphic mapping satisfying $\int_V f^*\omega \leq\delta_0$, then 
$\norm{df}_{L^\infty(K)}\leq C$.
\end{cor}

We now prove the theorem. The idea,  based on  an argument from  \cite{ds-pl} (see also \cite[Prop. 5.7]{fatou}), is to use a slicing argument together with  a theorem due to  Sibony and Wong   \cite{sibony-wong}. 
For convenience let us state this result first.

\begin{thm}[Sibony-Wong]\label{thm:sibonywong}
let $g$ be a holomorphic function defined in the neighborhood of the origin in $\cc^d$, which admits a holomorphic continuation to a neighborhood of $\bigcup_{L\in E} L\cap B(0,R)$, where $E\subset \pp^{d-1}$ is a set of lines through the origin,  of measure $\geq 1/2$
(relative to the Fubini-Study volume on $\pp^{d-1}$).

Then there exists a constant $C_{SW}>0$ such that $g$ extends to a holomorphic function of $B(0, C_{SW}R)$, and furthermore 
\begin{equation}\label{eq:sibonywong}
  \sup_{B(0, C_{SW}R)}\nolimits \abs{g} \leq
\sup_{\bigcup_{L\in E} L\cap B(0,R)}\nolimits \abs{g}.
\end{equation}
\end{thm}

\begin{proof}[Proof of Theorem \ref{thm:volume}] Recall that $\La$ was supposed to be an open subset in $\mathbb{C}^d$. Denote by $\beta$ the standard K\"ahler form on $\cc^d$. Let $T_n = f_n^*\omega$ and consider a subsequence (still denoted by $n$) such that $T_{n}$ converges to a current $T$ on $\La$. Let $\sigma_T = T\wedge \beta^{d-1}$ be the trace measure of $T$. 
 For every $p\in \La$,% $$\unsur{c_{2d-2}r^{2d-2}}\int_{B(p, r)} T\wedge \beta^{d-1} = \unsur{c_{2d-2}r^{2d-2}} \sigma_T(B(p, r)) \text{ (}\sigma_T\text{ the  trace measure}$$ 
$ \unsur{c_{2d-2}r^{2d-2}} \sigma_T(B(p, r))$
 converges as $r\cv0$ to the Lelong number of $T$ at $p$, denoted by $\nu(T, p)$ (here $c_{2d-2}$ is the volume of the unit ball in $\cc^{2d-2}$). By Siu's semi-continuity theorem \cite{demailly}, for each $c>0$, $E_c(T)= \set{p, \ \nu(T,c)\geq c}$ is a proper analytic subvariety of $\La$. Fix $c=\delta_0/4$, where $\delta_0$ is as above. We   show that $(f_n)$ is a normal family on $\La\setminus E_c(T)$. 
 
Indeed, let $p\notin  E_c(T)$. Then for  $r<r(\delta_0)$ (which will be fixed from now on), $$\unsur{c_{2d-2}r^{2d-2}}\int_{B(p, r)} T\wedge \beta^{d-1} \leq \frac{\delta_0}{3},$$ hence for large $n$, $$\unsur{c_{2d-2}r^{2d-2}}\int_{B(p, r)} T_{n}\wedge \beta^{d-1} \leq \frac{\delta_0}{2}.$$ Let now $\alpha_p = dd^c\log\norm{z-p}$. By Crofton's formula 
\cite[Cor. III.7.11]{demailly}, 
$\alpha_p^{d-1} = \int_{\pp^{d-1}}[L]dL$ is the average of the currents of integrations along the complex lines through $p$ (w.r.t. the unitary invariant probability measure on $\pp^{d-1}$). By a well-known formula due to Lelong 
\cite[\S 15.1]{chirka}, for every $r_1<r$, 
$$\unsur{c_{2d-2}r ^{2d-2}} \sigma_{T_{n}}(B(p, r ))-\unsur{c_{2d-2}r_1^{2d-2}} \sigma_{T_{n}} (B(p, r_1)) = 
\int_{r_1<\norm{z-p}< r}  T_{n}\wedge \alpha_p^{d-1}.$$ Since $T_{n}$ is a smooth form it has zero Lelong number at $p$ and we can let $r_1$ tend to zero. We conclude that for every large enough $n$, 
$\int_{B(p,r)}T_{n}\wedge \alpha_p^{d-1} \leq \frac{\delta_0}{2}$. 
 Applying Crofton's formula we see that 
$$\int_{\pp^{d-1}} \left(\int_{L\cap B(p,r)} f_{n}^*\omega\right) \leq \frac{\delta_0}{2},$$ therefore 
there exists  a set $E_n$ of lines of measure at least $1/2$ such that if $L\in E_n$ then
$\int_{L\cap B(p,r)} f_{n}^*\omega<\delta_0$. By Corollary \ref{cor:derivative} above, for each such line $L\in E_n$, the derivative of $f_{n}\rest{L\cap B(p,r)}$ is locally uniformly bounded. Extract a further subsequence so that $f_n(p)$ converges to some $x\in X$.
 Thus, reducing $r$ if necessary, we can assume that for $L\in E_n$, $f_{n}\rest{L\cap B(p,r)}$ takes its values in a fixed  coordinate patch containing $x$, which may be identified to a ball in $\cc^k$. By Theorem \ref{thm:sibonywong},  there exists a   constant $C_{SW}$   such that $f_n\rest{B(p,C_{SW}r)}$ takes its values in the chart, with the same bound on the derivative. This implies  that $(f_n)$ is a normal family in $B(p,C_{SW}r)$, thereby concluding the proof. 
\end{proof}

\begin{rmk}\label{rmk:normal}
 The proof   shows that if it can be shown that the Lelong numbers of the cluster values of $(f_n)^*\omega$ are smaller than $\delta_0$ (for instance if the potentials are uniformly bounded), then the family is actually normal. 
\end{rmk}
% 
%\begin{rmk}
% The assumption that $X$ is K\"ahler may certainly be dropped by replacing in general $\omega$ by a \textit{Gauduchon metric}, that is a positive (1,1) form with $dd^c\omega\geq 0$ and observing that Lelong numbers make sense for pluripositive currents, that is currents $T$ such that $T$ and $dd^cT$ are positive \cite{skoda}
%\end{rmk}
%
 
\subsection{Equidistribution in codimension 1}
We now turn to the case where the  mass of $f_n^*\omega$ tends to infinity, and show that the preimages of hypersurfaces under $f_n$ of $X$ tend to equidistribute in the sense of currents. This idea goes back to the work of Russakovskii, Shiffmann and Sodin \cite{russakovskii-sodin, russakovskii-shiffman}. Dinh and Sibony later gave
 in \cite{ds-tm} a wide generalization of these results.  
%and was pushed to   a very general context by Dinh and Sibony \cite{ds-tm}. 
Here we present a simple instance of this phenomenon, which is inspired by (and can be deduced from) \cite{ds-tm}.

\medskip

 Let $\La$, $X$ and $(f_n)$ be as above and set $d_n = \int_\La f_n^*\omega\wedge \beta^{d-1}$, so that $ {d_n\inv} f_n^*\omega $ is a sequence of currents of bounded mass on $\La$. A first remark is that if $\omega'$ is another K\"ahler form, there exists a constant $C\geq 1$ such that $ {C}\inv\omega\leq \omega'\leq C\omega$, hence $ {d_n}\inv f_n^*\omega'$   also has bounded mass. 

By definition, a holomorphic family of subvarieties $(H_a)_{a\in A}$ of dimension $l$
 parameterized by a complex manifold $A$ is the data of a subvariety $H$ in $A\times X$, of dimension $\dim (A)+l$, 
such that for every $a\in A$, $\pi_1\inv(a)\cap H =:H_a$ has dimension $l$.   Of course here  we are identifying every fiber $\pi_1\inv(a)$ with $X$, using the second projection. We will only need to consider the case $l=k-1$.

We need a notion of a ``sufficiently mobile" family of hypersurfaces. For this, let us assume for simplicity that $X$ is a projective manifold.  
We say that $(H_a)_{a\in A}$ is a {\em substantial} family of hypersurfaces on $X$ if  
the hypersurfaces $(H_a)$ are hyperplane sections relative to some embedding  $\iota:X\hookrightarrow \pp^N$ and there exists 
 a positive measure $\nu$ on $A$ such that the current  $\int[H_a] d\nu(a)$ has locally bounded potentials.  

Notice that the family of {\em all} hyperplane sections relative to some  projective embedding  of $X$ is substantial. Indeed by Crofton's formula there exists a natural smooth measure $dL$ on the dual projective space $\check\pp^N$ (i.e. the space of hyperplanes) such that 
the average current of integration is the Fubini-Study form, i.e. $\int[L]dL = \omega_{FS}$. Therefore on $X$ we get that $\int [\iota\inv(L)] dL = \iota^*\omega_{FS}$, and the family is substantial. 
From this it follows for instance that   the family of hypersurfaces of given degree in $\pp^k$ is substantial. 

\medskip

Here is the equidistribution statement. We do not strive for maximal generality here and it is likely that  some of  the assumptions could be relaxed. For instance, in view of applications to random walks on groups it is of interest to obtain similar results   for non compact $X$ (to deal with examples like $X=\mathrm{SL}(n,\cc)$, etc). One  might also obtain equidistribution statements  in higher codimension, by introducing appropriate dynamical degrees. 

\begin{thm}\label{thm:equidist abstrait}
Let $\La$ be a complex manifold of dimension $d$ and $X$ be a projective  manifold. 
 Let $f_n:\La\cv X$ be a family of holomorphic mappings such that $d_n = \int_\La f_n^*\omega\wedge \beta^{d-1}$ tends to infinity. Let $(H_a)_{a\in A}$ be an substantial holomorphic family of hypersurfaces in $X$.  

Let $\mathcal{E}  $ be the set of $a\in A$ such that 
%defined by the property that  if $a \notin \mathcal E$, the sequence 
$$\unsur{d_n}(f_n^*[H_a] - f_n^*\omega) \text{ does not converge to zero in the sense of currents.}$$
Then:
\begin{enumerate}
 \item[i.] if the series $  \sum {d_n\inv}$ converges, then  $\mathcal{E}$ is pluripolar;
 \item[ii.] if for every $t>0$ the series $\sum \exp(-td_n)$ converges, then $\mathcal{E}$ has zero Lebesgue measure.
\end{enumerate}
\end{thm}
 
As the proof easily shows, in case {\em ii.} Lebesgue measure can actually  be replaced by any {\em moderate  measure}, that is a measure  
satisfying an inequality of the form $m(\set{u<-t})\leq C e^{-\alpha t}$ on any compact class of psh functions. This is wide class of measures which contains for instance the area measure on totally real submanifolds of maximal dimension.  We refer to \cite{ds-tm} for details.

\begin{cor}
Under the assumptions of the theorem,
if  the sequence  $\unsur{d_n} f_n^*\omega$ converges to some  current $T$, then  for $a\notin \mathcal{E}$, $\unsur{d_n}f_n^*[H_a] $ converges to $T$. 
\end{cor}

\begin{proof}[Proof of Theorem \ref{thm:equidist abstrait}]
 Without loss of generality, assume that $A$ is an open ball in $\cc^{\dim(A)}$. 

\begin{lem}\label{lem:uax}
 Under the assumptions of the theorem, there exists a K\"ahler form $\omega$ on $X$ and a negative function $(a,x)\mapsto u(a,x)$ on $A\times X$ such that 
 \begin{enumerate}[i.]
 \item for every $a\in A$,  $dd^c_x u(a, \cdot) = [H_a]-\omega$;
 \item the $L^1$ norm $\norm{u(a, \cdot)}_{L^1(X)}$ is locally uniformly bounded with respect to  $a\in A$;
 \item for every $x$, $  u(\cdot,x)$ is psh on $A$.
\end{enumerate}
\end{lem}

Assuming this result for the moment, let us continue with  the proof of the theorem. Suppose first that the series $  \sum {d_n\inv}$ converges. Let $m$ be a positive measure with compact support on $A$ such that psh functions are $m$-integrable. We claim that for $m$-a.e. $a$, $\unsur{d_n}(f_n^*[H_a] - f_n^*\omega)$ converges to zero. By Lemma \ref{lem:uax}.i. for this  
 it is enough to prove that $\unsur{d_n}\int u_a\circ f_n$ tends to zero in $L^1_{\rm loc}(\La)$ (here and in what follows we denote $u(a, \cdot)$ by $u_a$). 

Let us admit  the following    lemma for the moment.

\begin{lem}\label{lem:bdd}
The function defined by $x\mapsto \int u_a(x) dm(a)$ is $\omega$-psh (i.e. $dd^cu_a\geq -\omega$) and bounded on $X$.
\end{lem}

  Fix  
$\La'\Subset \La$. By Fubini's theorem and Lemma \ref{lem:bdd}, setting $ \widetilde{u}(x) = \int u_a(x) dm(a)$ we have that  
$$\int\lrpar{\unsur{d_n}\int_{\La'} (-{u_a\circ f_n(\la)}) d\la}dm(a) = \unsur{d_n}\int_{\La'}  (-{\widetilde{u}\circ f_n(\la)})d\la \leq \frac{C }{d_n}, $$
so 
$$m\lrpar{\set{ a, \ \unsur{d_n}\int_{\La'} \abs{u_a\circ f_n(\la)} d\la \geq \e}}\leq \frac{C }{\e d_n},$$ and the result follows from the Borel-Cantelli Lemma. 

\medskip

To conclude the proof of case {\em i.} in the theorem, we argue  that if the exceptional $\mathcal{E}$ set was not pluripolar, then it would contain a non-pluripolar compact set $K$. By the work of Bedford and Taylor \cite{bt} there exists a Monge-Ampère measure $m=(dd^cv)^N$ supported on $K$, with $v$ bounded. It is well known that for such a measure psh functions are integrable (see \cite{bt} or  \cite[Prop. III.3.11]{demailly}), so we are in contradiction with the previous claim.  

\medskip

Assume now that $\sum \exp(-td_n)$ converges for all $t$. By  Lemma \ref{lem:bdd} applied to $m$   the Lebesgue measure (cut-off to any compact subset of $A$) 
the family  of negative psh functions 
$$\set {a\mapsto \int_{\La'} {u_a\circ f_n(\la)} d\la}_{n\geq 1}$$ is bounded in $L^1_{\rm loc}(A)$. 
Let $A'\Subset A''\Subset A$. It follows from an inequality due to Hörmander 
\cite[Prop 4.2.9]{hormander}
 that there exists constants $(C_0,\alpha_0)$  such that if $\varphi$ is any negative psh function on $A$ such that $\norm{\varphi}_{L^1(A'')}\leq 1$, then 
$$\vol\lrpar{\set{a\in A', \ \varphi(a)<-t}} \leq C_0\exp(-\alpha_0 t).$$
It follows that  there exist constants $(C,\alpha)$ independent of $n$ such that 
$$\vol\lrpar{\set{ a\in A', \ \unsur{d_n}\int_{\La'} \abs{u_a\circ f_n(\la)} d\la >\e}}\leq C \exp( -\e \alpha d_n), $$ and again the  Borel-Cantelli implies that for (Lebesgue) a.e. $a$, 
$\unsur{d_n} u_a\circ f_n$ converges to zero in $L^1_{\rm loc}(A)$.
\end{proof}

\begin{proof}[Proof of Lemma \ref{lem:uax}]
%Recall that $H_a$ is the zero set of a holomorphic section $s_a$ of a fixed ample line bundle $L$, depending holomorphically on $a$.  Now, $L$ admits a hermitian metric $\abs{\cdot}_h$ with positive curvature, and we simply put $u(a,x) = dd^c\log\abs{s_a(x)}_h$. In a  trivializing chart $U_i$ of $L$, $s_a$ expresses as $s_{a,i}(x)$, where $s_{a,i}$ depends holomorphically on $a$, and 
%$\abs{s_a(x)}_h = \abs{s_{a,i}(x)}\exp(-h_i(x))$, where $h_i$ is a strictly psh function. Hence 
%$u(a,x) = \log\abs{s_{a,i}(x)} - h_i$, which clearly satisfies our requirements.
By definition there is an embedding $\iota:X\hookrightarrow \pp^N$ and a holomorphic family
$(L_a)_{a\in A}$ of hyperplanes such that $H_a=\iota\inv (L_a)$.  There exists a holomorphic family of linear forms $(\ell_a)_{a\in A}$ on $\cc^{N+1}$ such that $\set{\ell_a=0}$ is an equation of $L_a$. 
We normalize so that $\abs{\ell_a}\leq 1$ on the unit ball.
Now define $\varphi(a,\cdot)$ on $\pp^N$  by $$\varphi(a,z) = \frac{\log\abs{\ell_a(Z)}}{\log\norm{Z}}, $$ where $Z$ is any lift of $z$ and $\norm{\cdot}$ is the Hermitian norm. Then $\varphi$ satisfies  {\em i.}-{\em iii.} relative to the family $L_a$ on $\pp^N$, i.e. $dd^c_z \varphi(a, \cdot) = [L_a]-\omega_{FS}$, etc. 

We now put $u(a,x) = \varphi(a, \iota(x))$ and claim that it satisfies the desired requirements (with $\omega = \iota^*\omega_{FS}$).  Properties  {\em i.} and {\em iii.} are immediate.  If by contradiction {\em ii.} did not hold, then by the Hartogs Lemma \cite[pp. 149-151]{hormander} we would get a sequence $a_n\cv a_0\in A$ such that $u_{a_n}$ diverges uniformly to 
$-\infty$. But if $x_0\notin H_{a_0}$ it is clear that $u$ is locally uniformly bounded near $(a_0, x_0)$, hence the contradiction. 
\end{proof}

\begin{proof}[Proof of Lemma \ref{lem:bdd}] 
Let $A'\Subset A$ be an open set containing $\supp(m)$
According to \cite[Prop. 3.9.2]{ds-pl} there exists a constant $C$ such that for every psh function $\varphi$ on $A$, $\norm{\varphi}_{L^1(m)}\leq C\norm{\varphi}_{L^1(A')}$. From this we infer that 
for every $x$ in $X$, $\int \abs{u(a,x)}dm(a)\leq C\int_{A'} \abs{u(a,x)}da$, where $da$ denotes the Lebesgue measure.  

Now we claim that there exists a constant $C'$ such that for any negative psh function on $A$, 
$\norm{\varphi}_{L^1(A')}\leq C' \norm{\varphi}_{L^1(\nu)}$, where $\nu$ is the measure from the definition of substantial families. Indeed, by the Hartogs lemma (see \cite[pp.149-151]{hormander}) the set
$$\set{\varphi \text{ negative psh s.t. }\int\abs{\varphi}d\nu\leq1}$$ is relatively compact in $L^1_{\rm loc}$, hence bounded.  

From these two facts we infer that for every $x$, 
 $$\int \abs{u(a,x)}dm(a)\leq CC'\int \abs{u(a,x)} d\nu(a).$$ We now show that $x\mapsto \int u_a(x) d\nu(a)$  is uniformly bounded.  For this, note first that   this function  is   integrable, because by Lemma \ref{lem:uax}, $\norm{u(a, \cdot)}_{L^1(X)}$ is locally uniformly bounded. So we can take the $dd^c$ in $x$ and we obtain that 
 $$dd^c_x\lrpar{\int u_a(\cdot) d\nu(a)} = \int [H_a]d\nu(a) - \omega 
 ,$$ and we conclude by using our assumption that the local potentials of $\int [H_a]d\nu(a)$ are bounded. Finally, the same argument shows that   $x\mapsto \int u_a(x) dm(a)$ is $\omega$-psh, since by uniform boundedness of $\norm{u(a, \cdot)}_{L^1(X)}$, we can permute the $dd^c$ in $x$ and  integration with respect to $m$. 
 \end{proof}

We close this section by  highlighting  the following direct consequence of Theorem \ref{thm:volume}, which appears as a   key step   in the characterization  of the supports of certain bifurcation   currents.

 It can also be used to obtain a coordinate-free proof of the fact that the support of the  Green 
  current of a holomorphic self-map of  $\pk$ coincides with the Julia set (a result originally due to Forn\ae ss-Sibony \cite{fs2} and Ueda \cite{ueda}). 
 More precisely, it implies that if $\om\subset \pk$ is an open set disjoint from the support of the Green 
 current $T$, then the sequence of iterates $f^n$ is normal 
in $\om$. Notice that 
the converse inclusion $\supp(T)\subset  {J}$ is an  easy consequence of the definitions.

\begin{prop}\label{prop:support}
Let $\La$ be a complex manifold of dimension $d$ and $X$ be a compact K\"ahler manifold of dimension $k$, endowed with a K\"ahler form  $\omega$. Let $(f_n)$ be a sequence of holomorphic mappings $\La\cv X$, and assume that the sequence   
$\unsur{d_n}  f_n^*\omega$ converges to a current $T$. 

Assume  furthermore that  every test function $\varphi$ one has the following estimate
\begin{equation}\label{eq:mass}
\bra{ \unsur{d_n} f_n^*\omega -T, \varphi} = O\lrpar{\unsur{d_n}}.
 \end{equation}

Then the sequence $(f_n)$ is quasi-normal outside $\supp(T)$. 
\end{prop}

\begin{proof}
If $U$ is an open set disjoint where $T\equiv 0$, then \eqref{eq:mass} implies that the sequence $(f^*_n\omega)$ has locally uniformly bounded mass on $U$. It then follows from Theorem \ref{thm:volume} that $(f_n)$ is quasi-normal on $U$. 
\end{proof}

Notice that conversely, if $d_n\cv\infty$, and $U$ is an open set such that $U\cap \supp(T)\neq \emptyset$, then the sequence $(f_n)$ is {\em not} normal on $U$. Indeed it follows from the explicit expression of $f^*_n\omega$ in local coordinates that the $L^2$ norm of the derivative of $f_n$ tend to infinity in $U$.

%%%%%%%%%%%%%%%%%%%%%%%%%%%%%%%%%%%%%%%%%%%%%%%%

\section{Bifurcation currents for families of rational mappings on $\pu$}\label{sec:bifcur}

\subsection{Generalities on bifurcations}
 Let us first review a number well-known facts on holomorphic families of rational maps. 
Let   $(f_\la)_{\la \in \La}$ be a holomorphic family of rational maps  $f_\la:\pu\cv\pu$ of degree $d\geq 2$
parameterized by a connected complex manifold $\La$. By definition, a {\em marked critical point} is a holomorphic map $c:\La\cv\pu$ such that  $f_\la'(c(\la))=0$ for all $\la\in \La$.

Given any family $(f_\la)$ if $c_0$ is a given   critical point at parameter $\la_0$, there exists a branched cover $\pi: \widetilde \La\cv\La$ such that the family of rational mappings $(\widetilde f_{\widetilde\la})_{\widetilde \la\in \widetilde\La}$ defined for $\widetilde \la\in 
\widetilde\La$ by $\widetilde f_{\widetilde\la} = f_{\pi(\widetilde\la)}$ has a marked critical point $\widetilde c(\widetilde\la)$ with $\widetilde c(\widetilde \la_0)=c_0$. Specifically, 
 it is enough to parameterize the family by $\widetilde{\La} = 
\widehat{\mathcal C} = \set{(\la, z)\in \La\times \pu, \ f'_\la(z)=0}$ (or its desingularization if $\widehat{\mathcal C}$ is not smooth). Then the first projection $\pi_1: \widehat{\mathcal C} \cv \La$ makes it   a branched cover over $\La$, and for any  $\widetilde \la = (\la,z)\in \widehat{\mathcal C}$, we set $\widetilde c(\widetilde\la)= z$ which
 is  a critical point for $\widetilde f_{\widetilde\la} := f_{\pi_1(\widetilde\la)}$. 

Therefore, 
taking a branched cover of $\La$ if necessary,  it is always possible to assume that all critical points are marked. 

We always denote with a subscript $\la$ the dynamical objects associated to $f_\la$:   Julia set,   maximal entropy measure, etc.

In a celebrated paper, Mañé, Sad and Sullivan \cite{mss}, and independently  Lyubich \cite{lyubich-bif},  showed the existence of a decomposition $$\La = \bif\cup\stab$$  of the parameter space $\La$ into a (closed) bifurcation locus and a (open) stability locus, which is similar to the Fatou-Julia decomposition of dynamical space. 

Stability is defined by a number of equivalent properties, according to the following theorem \cite{mss, lyubich-bif}.

\begin{thm}[Ma\~né-Sad-Sullivan, Lyubich]\label{thm:mss}
Let $(f_\la)_{\la \in \La}$  be a holomorphic family of rational maps of degree $d\geq 2$ on $\pu$. Let $\om\subset\La$  be a connected open subset. The following conditions are equivalent:
\begin{enumerate}[i.]
\item the periodic points of  $(f_\la)$  do not change nature (attracting, repelling, indifferent) in  
$
\om$;
\item the Julia set $J_\la$ moves under a holomorphic motion for $\la\in \om$;
\item  $\la\mapsto J_\la$ is continuous for the Hausdorff topology;
\item for any two parameters 
 $\la,\la'$ in $\om$, $f_\la\rest{J_\la}$  is conjugate to $f_{\la'}\rest{J_{\la'}}$.
\end{enumerate}
If in addition,  the critical points  $\set{c_1,\ldots , c_{2d-2}}$  are marked, these conditions are equivalent to 
\begin{enumerate}[ v.]
\item for any $1\leq i\leq 2d-2$ the family of meromorphic functions 
 $(f_\la^n(c_i(\la)))_{n\geq 0}: \La\cv \pu$  is normal
\end{enumerate}
\end{thm}

If these conditions are satisfied, we say that $(f_\la)$ is {\em  $J$-stable} in $\om$ (which we simply abbreviate   as {\em stable} in this paper). The stability locus is the union of all such $\om$, and $\bif$ is by definition
 its complement.
 
Another famous result is the following  \cite{mss, lyubich-bif}.

\begin{thm}%[Ma\~né-Sad-Sullivan, Lyubich] 
\label{thm:dense}
Let $(f_\la)_{\la\in\La}$ be as above.  
The stability locus $\stab$  is dense in $\La$.
\end{thm}
 
 If $(f_\la)$ is the family of all polynomials or all rational functions of degree $d$, it is conjectured that $\la\in \stab$ if and only if all critical points converge to attracting cycles (the {\em hyperbolicity conjecture}). More generally, the work of McMullen and Sullivan \cite{mcms} leads to a conjectural description of the components of the stability locus in any  holomorphic family of rational maps (relying on the  the so-called {\em no invariant line fields conjecture}).  
 
 \medskip 
 
We now explain how the  bifurcation locus can be seen in a number of ways as the  limit  set (in the topological sense) of countable families of analytic subsets of codimension 1. This is  a basic motivation for a description of the bifurcation locus in terms of positive closed currents of bidegree (1,1).

 A marked critical point $c$ is said to be {\em passive} in $\om$ if the family $(f_\la^n(c(\la)))_{n\geq 0}$ is normal, and {\em active} at $\la_0$ if for every neighborhood $V\ni\la_0$, $c$ is not passive in $V$ (this convenient  terminology is due to McMullen).  The characterization  {\em v.} of stability in Theorem \ref{thm:mss}  then rephrases as ``a critically marked family is stable iff all critical points are passive in $\om$".

\medskip

A typical example of a passive critical point is that of a critical point converging  to an attracting cycle, for this property is robust under perturbations. If $\La$ is the family of all polynomials or all rational functions of degree $d$, according to the hyperbolicity conjecture, all passive critical points should be of this type. Indeed, any component of passivity would intersect the hyperbolicity locus. On the other hand, in a family of rational mappings with a persistent parabolic  point (resp. a persistent Siegel disk), a critical point attracted by this parabolic point (resp. eventually falling in this Siegel disk) is passive. 

Let $(f_\la)$ be the space of all polynomials or rational mappings of degree $d$ with a marked critical point $c$. 
It seems to be an interesting problem to study the geometry of  {\em hyperbolic passivity components}, that is, components of the passivity locus associated to $c$, where $c$ converges to an attracting cycle. Does there exist a ``center'' in this component, that is a subvariety where $c$ is periodic? How does the topology of the component related to that of its center? 

\medskip

 The following result is an easy consequence of  Montel's theorem (see e.g. \cite{levin1} or \cite{preper} for the proof). 
  
 \begin{thm}\label{thm:preperiodic dense}
 Let $(f_\la)_{\la \in \La}$  be a holomorphic family of rational maps of degree $d$ on $\pu$, with a marked critical point $c$. If $c$ is active at $\la_0$ then $\la_0 = \lim_n \la_n$ where for every $n$,  
 $c(\la_n)$ is periodic   (resp. falls  onto a repelling cycle).
 \end{thm}

Combined with  Theorem \ref{thm:mss} this implies that in any (not necessarily critically marked) holomorphic family, the  family of  hypersurfaces, defined by the condition that a critical point is periodic (resp. preperiodic) cluster on the bifurcation locus.  

\begin{cor}
 Let $(f_\la)_{\la \in \La}$  be a holomorphic family of rational maps of degree $d$ on $\pu$. Then 
$$\bif   \subset\overline{\set{\la, \ \exists   \ c(\la) \text{ \rm (pre)periodic critical point} }}$$ and more precisely, 
$$ \mathrm{Bif}=\overline{\set{\la, \ \exists   \ c(\la) \text{ \rm  critical point falling non-persistently on a repelling cycle}}}.$$
\end{cor}

At this point a natural question arises: assume that several marked critical points are active at $\la_0$. Is it possible to perturb $\la_0$ so that these critical points become simultaneously (pre)periodic? 
It turns out that the answer to the question  is ``no'', which is a manifestation of the failure of Montel's theorem in higher dimension (see Example \ref{ex:douady} below). An important idea in higher dimensional holomorphic dynamics is that the use of currents and pluripotential theory is a way to get around this difficulty. As it turns out, the theory of bifurcation currents  will indeed  provide  a reasonable understanding of this problem. 

\medskip

The following simple consequence of item {\em i.} of Theorem \ref{thm:mss}, provides yet another dense codimension 1 phenomenon in the bifurcation locus. 

 \begin{cor}
 Let $(f_\la)_{\la \in \La}$  be a holomorphic family of rational maps of degree $d$ on $\pu$. Then 
for every $\theta\in \re/2\pi\zz$,
$$\mathrm{Bif}=\overline{\set{\la, \ f_\la \text{ admits a non-persistent periodic point of multiplier }e^{i\theta}}}.$$
\end{cor}

Again one  might ask: what is the set of parameters possessing \textit{several} non-persistent indifferent periodic points?

\subsection{The bifurcation current}\label{subs:bifcur}

It is a classical observation that in a holomorphic family of dynamical systems, Lyapunov exponents often depend subharmonically on parameters  (this  idea  plays for instance a key role in \cite{herman}). Also,  since the 1980's,
potential theoretic methods appear to play an important role  in the study the quadratic family and the Mandelbrot set (see e.g. \cite{dh, sibony-ucla}). 

DeMarco made this idea more systematic by putting forward the following definition \cite{demarco2}. 

\begin{propdef} 
Let $(f_\la)_{\la\in \La}$ be a holomorphic family of rational maps of degree $d\geq 2$.  For each $\la\in \La$, let
 $\mu_\la$ be the unique   measure of maximal entropy of $f_\la$. 
Then 
$$\chi:\la\longmapsto \chi(f_\la) =  \int \log\norm{ f_\la  ' }d\mu_\la$$ is a continuous psh function on  $\La$ (here the norm of the differential is relative  to   any Riemannian metric on $\pu$). 

The {\em bifurcation current}   of the family $(f_\la)$ is by definition $\tbif = dd^c\chi$.
\end{propdef}

The continuity of $\chi$ was originally proven by  Mañé \cite{mane}. Actually  $\chi$ is  Hölder continuous, as can easily be seen from DeMarco's formula for $\chi$ (see below), and the joint Hölder continuity in $(\la,z)$ 
of the dynamical Green's function (see also \cite[\S 2.5]{ds-survey} for another approach). 
 
The significance of this definition is justified by the following result \cite{demarco1, demarco2}.  
 
\begin{thm}[DeMarco]
The support of $\tbif$ is equal to $\bif$.
\end{thm}

\begin{proof}[Proof (sketch).]
 The easier inclusion is the fact that $\supp(\tbif)\subset\bif$, or equivalently, that $\chi$ is pluriharmonic on the stability locus. A neat way to see this is to use the following approximation formula, showing that the Lyapunov exponent of $f_\la$ can be read on periodic orbits: for every rational map $f$ of degree $d$,
\begin{equation}\label{eq:bdm}
\chi(f ) = \lim_{n\cv\infty} \unsur{d^n}\sum_{p \text{ repelling of   period } n}  \unsur{n} \log^+\abs{(f^n)'(p)} 
\end{equation}
(see Berteloot \cite{berteloot parma} for the proof). It thus follows from the characterization {\em i.} of Theorem \ref{thm:mss} that if $(f_\la)$ is stable on some open set $\om$, then $\chi$ is pluriharmonic there (notice that by the Hartogs Lemma the pointwise limit of a uniformly bounded sequence of pluriharmonic functions is pluriharmonic).

\medskip

Let us sketch DeMarco's argument for the converse inclusion. It is no loss of generality to assume that all critical points are marked. The main ingredient is a formula relating $\chi$ and the value of the Green's function at critical points (see also Proposition \ref{prop:sum ci} below). We do not state this formula   precisily here\footnote{The expression  of DeMarco's formula for a rational map $f$ is of the form $\chi(f) = \sum_{i=1}^{2d-2} G_F(c_j) + H(f)$, where $G_F$ is the dynamical Green's function of a homogeneous lift $F:\cd\cv\cd$ of $f$, the $c_j$ are certain lifts of the critical points, and $H(f)$  depends pluriharmonically on $f$}, and only  note that it generalizes the  well-known formula due to Przytycki \cite{prz} (see also Manning \cite{manning}) for the Lyapunov exponent of a monic polynomial $P$:
\begin{equation}\label{eq:lyap}
 \chi(P)  = \log d + \sum_{c \text{ critical}} G_P(c)   .
\end{equation}
Here $G_P$ denotes the dynamical Green's function of $P$ in $\cc$, defined by 
$$G_P(z) = \lim_{n\cv\infty} d^{-n}\log^+\abs{P^n(z)}.$$ 

From this one deduces  that if $\chi$ is pluriharmonic in some open set $\om$, then all critical points are passive in $\om$, hence the family is stable  by  Theorem \ref{thm:mss}. 
\end{proof}

A first consequence of this result, which was a   source of motivation in \cite{demarco1}, is that if $\La$ is a Stein manifold (e.g. an affine algebraic manifold), then the components of the stability locus are also Stein.

\medskip

In the most studied  quadratic family $(P_\la(z))= (z^2+\la)_{\la\in \cc}$, we see from \eqref{eq:lyap} that the bifurcation  ``current''  (which is simply a measure in this case)
 is   defined by the  formula  $\mubif  = dd^c G_{P_{\la}}(0)$, where $G_{P_\la}$ is the Green  function. As expected, we recover the usual parameter space measure, that is 
 the harmonic measure of the Mandelbrot set.

\subsection{Marked critical points}\label{subs:marked}
In this paragraph we present a construction due to Favre and the author \cite{preper} of a 
 current associated to the bifurcations of a marked critical point.  It would also be possible to define this current  by lifting the dynamics to $\cd\setminus \set{0}$ and evaluating appropriate dynamical Green's function at the lifted critical points, in the spirit of \cite{demarco1}. However our construction is more instrinsic, and generalizes to other situations. 

Let $(f_\la,c(\la))$ be a holomorphic family of rational maps of degree $d\geq 2$ with a marked critical point. 
Let $f_n:\La\cv\pu$ be defined by $f_n(\la) = f_\la^n(c(\la))$. Let $\omega$ be a Fubini-Study form on  $\pu$.

In the spirit of Section \ref{sec:prologue} we have the following result.

\begin{thm}\label{thm:activity support}
Let  as above $(f,c)$  be a holomorphic family of rational maps of degree $d\geq 2$ with a marked critical point, and  set  $f_n: \la\mapsto f_\la^n(c(\la))$. 
Then the  sequence of currents $(d^{-n}f_n^*\omega)$ converges to a current $T_c$ on $\La$. The support of $T_c$ is the activity locus of $c$. 
\end{thm}

By definition,  $T_c$ is \textit{the bifurcation current} (also referred to as the {\em activity current}) associated to $(f,c)$.

\begin{proof}[Proof (sketch)]
The convergence of the sequence of currents 
 $(d^{-n}f_n^*\omega)$ does not follow from the general formalism of Section \ref{sec:prologue}. The proof relies  on   equidistribution  results for preimages of points under $f^n$ instead. For this, it is convenient to   consider the product dynamics on $\La\times \pu$. 
Let  $\hL = \Lambda \times \pu$. The family $f_\lambda$ lifts to a holomorphic map
$\hf : \hL \to \hL$ mapping   $(\lambda, z) $ to $(\lambda,
f_\lambda(z))$.  We denote by  $\pi_1: \hL \cv \Lambda$ and $\pi_2: \hL
\to \pu$ the natural projections, and let $\widehat \omega = \pi_2^*\omega$. 

The following     proposition follows from classical techniques in higher dimensional holomorphic dynamics.

\begin{prop}  
 Let    $(f_\la)$  be a holomorphic family of rational maps of degree $d\geq 2$, and $\hf$ be its lift to $\La\times \pu$ as above. Then the sequence of currents  $d^{-n}\hf^{n*} \widehat\omega$ converges to a limit  $\hT$ in  $\La\times \pu$.
\end{prop}

The current $\hT$ should be understood as ``interpolating'' the family of maximal measures $\mu_\la$.

Let  $\widehat c = \{ (\lambda, c(\lambda)), \ \la \in \La\}\subset \hL$ (resp. $\hf^n(\widehat c)$) be the graph of $c$ (resp. $f^n(c)$). As observed in Lemma \ref{lem:dim1}, $d^{-n}f_n^*\omega = (\pi_1)_* \lrpar{d^{-n} \widehat\omega\rest{\hf^n(\widehat c)}}$.
Now $\hf^n$ induces a biholomorphism  $\widehat c \cv \hf^n(\widehat c)$, so 
$d^{-n} \widehat\omega\rest{\hf^n(\widehat c)} = d^{-n}\big((\hf^n)^*\widehat\omega\big)\rest{ \widehat c}$.
Thus we see  that the bifurcation current $T_c$  is obtained by slicing  $\hT$ by the hypersurface 
$\widehat c$ and projecting down to $\La$: $T_c = (\pi_1)_*\lrpar{\hT\rest{\widehat c}}$.

Making this precise actually requires
 a precise control on the convergence of the sequence $(d^{-n}(\hf^{n})^* \widehat\omega)$  to  $\hT$.  
This  follows from the following classical computation:   write   $ {d}\inv\hf ^*
 \hom  - \hom = dd^c g_1$. Then 
 $  d^{-n}(\hf^n )^* \hom  - \hom = dd^c g_n$, where $g_n = \sum_{j=0}^{n-1} d^{-j} g_0\circ \hf^j$. Therefore $(g_n)$ converges uniformly to   $g_\infty$, with $\hom + dd^c g_\infty = \hT$.  
 
 In particular we have that 
$\abs{g_n - g_\infty} = O(d^{-n})$, which implies that the assumption 
 \eqref{eq:mass} in Proposition \ref{prop:support} is satisfied.
Hence the family $(f^n_\la(c(\la))$ is quasi-normal outside $T_c$.  To see that it is actually normal, 
we notice that the  
  uniform control on the potentials  allows to apply Remark \ref{rmk:normal}.

\medskip

Conversely, $\supp(T_c)$ is contained in the activity locus. Indeed it follows
 from the explicit expression of $d^{-n}f_n^*\omega$ in local coordinates on 
$\La$ that if $U\subset\La$ is an open set intersecting $\supp(T_c)$,
 the $L^2$ norm of the derivative of $f_n$ (relative to the spherical metric on $\pu$)  grows exponentially in $U$. The result follows. 
  \end{proof}
   
Observe  that the fact that $c$ is a critical point does not play any role here. We might as well associate activity/passivity  loci and a bifurcation current to any holomorphically moving point $(a(\la))$, and Theorem \ref{thm:activity support} holds in this case (this type of considerations appear   e.g. in \cite{baker-demarco}).

\medskip

For the quadratic family $(P_\la(z)) = (z^2+\la)_{\la\in \cc}$, which has a  marked critical point at 0, one easily checks that the current $\hT$ is defined in $\cc\times \cc$ by the formula 
$\hT = dd^c_{(\la,z)} G_{P_\la}(z)$, and that  the bifurcation  current  associated to the critical point is again $dd^c_{\la} G_\la(0)=\mubif$.   

\medskip

More generally,  DeMarco's formula for the Lyapunov exponent of a rational map gives
the relationship between these currents and the bifurcation current $\tbif$ defined in  \S\ref{subs:bifcur}.
 
\begin{prop}\label{prop:sum ci}
Let  $(f_\la)$ be a family of rational maps with   all critical points marked $\set{c_1, \ldots, c_{2d-2}}$, then $\tbif = 
\sum T_i$, where $T_i$ is the bifurcation current associated to  $c_i$. 
\end{prop}

\subsection{Equidistribution of critically preperiodic parameters}
We keep hypotheses as before, that is we work  with a family of rational maps with a marked critical point $(f,c)$. Our goal is to  show that the bifurcation current $T_c$ is the limit in the sense of currents of sequences of dynamically defined codimension 1 subvarieties. 

The first result follows directly from Theorem \ref{thm:equidist abstrait}. It  is  a quantitative version of  the fact that near an activity point, $f^n_\la(c(\la))$ assumes almost every value in $\pu$. 

\begin{thm}\label{thm:nanamarre}
 Let $(f_\la, c(\la))_{\la\in \La}$ be a holomorphic family of rational maps of degree $d\geq 2$  with a marked critical point, and $T_c$ be the associated bifurcation current. 
There exists a pluripolar  exceptional set $\mathcal{E}\subset \pu$ such that if $a\notin  \mathcal{E}$, then 
$$\lim_{n\cv\infty} \unsur{d^n} [H_{n,a}] \cv T_c, \text{ where } H_{n,a}=\set{\la, \ f^n_\la(c(\la)) = a}.$$
\end{thm}

Note that $H_{n,a}$ is defined not only as a set, but as an analytic subvariety, with a possible multiplicity. It is likely that the size of the exceptional set can be estimated more precisely. 

\begin{question}
Is the exceptional set in Theorem \ref{thm:nanamarre} finite, as in the case of a single mapping? 
\end{question}

\medskip

It is dynamically  more significant to study the distribution of parameters for which $c$ becomes periodic (resp. preperiodic), that is, to try to make Theorem \ref{thm:preperiodic dense} an  equidistribution result. This is expected to be more difficult because in this case  the  set of targets that $f^n_\la(c(\la))$ is supposed to hit (the set of periodic points, say) is both countable and moving with $\la$. 

Let $e\in \set{0,1}$ be the cardinality of the exceptional set of $f_\la$ for generic $\la$. If $e=2$, then the family is trivial. If $e=1$, it is conjugate to a family of polynomials. Given two integers $n>m\geq 0$ denote by  $\percrit(n,m)$  the subvariety of $\La$ defined by the (non necessarily reduced) equation
 $f^n_\la(c(\la)) = f^m_\la(c(\la))$.

It is convenient to adopt the convention that $[\La]=0$. This means that if some subvariety $V$ like $\percrit(n,k)$ turns out to be equal to $\La$, then  we declare that $[V] = 0$.  

 The following equidistribution theorem was obtained in \cite{preper}.
 
\begin{thm}[Dujardin-Favre] \label{thm:cvg preper}
Let $(f,c)$ be a non-trivial holomorphic family of rational maps on  $\pu$   of degree $d\geq 2$, with a marked critical point, and  denote by $e$   the generic cardinality of the exceptional set. Assume furthermore that the following technical assumption is satisfied:
\begin{itemize}
\item[{\rm (H)}]{\it for every $\lambda \in \Lambda$, there exists an immersed curve  $\Gamma 
\subset \Lambda$  through 
    $\lambda$ such that the complement of the set  $\{ \lambda,
    \, c(\lambda) \text{is attracted by a periodic cycle}\}$ is relatively compact in   $\Gamma$.}
\end{itemize}
Then for every sequence $0\le k(n) <
n$, we have that
 $$\lim_{n\cv\infty} \frac{[\percrit(n,k(n))]}{d^n+d^{(1-e)k(n)}}  = T_c$$
\end{thm}

 Notice that with our convention,  if $c(\la)$ is periodic throughout the family, then  both sides of the equidistribution equation  vanish. 

%It is an   interesting  open question whether the  additional assumption (H) is really necessary. 
\begin{question}
Is assumption (H) really necessary?
\end{question} 

One might at least try to replace it by a more tractable  condition like  $\La$ being an algebraic family --compare with Theorem \ref{thm:equidist speed}. 
It is easy to see that (H) holds e.g. in the space of all polynomial  or rational maps of degree $d$ (see \cite{preper}). %It was shown by Bassanelli and Berteloot \cite{bas-ber2} that (H) is not necessary in the case where $k(n)= 0$ (see Theorem \ref{thm:basber} below). 

\medskip 

In the 1-parameter family $(z^d+\la)_{\la\in\cc}$ of unicritical polynomials of degree $d$,
the theorem implies the equidistribution of the centers of components of the degree $d$ Mandelbrot set  
 \begin{equation}\label{eq:equidist mandelbrot}
  \lim_{n\cv\infty} \unsur{d^n} [\percrit(n,0)]  = 
\lim_{n\cv\infty} \frac1{d^n}\sum_{f_c^n(0)=0} \delta_c = \mubif.
\end{equation}
This result had previously been proven by Levin \cite{levin} (see also McMullen \cite{mcm equidist}). 

Another interesting approach to this type of  equidistribution statements is to use 
arithmetic methods based on height theory (following work of Zhang, Autissier, Chambert-Loir, Thuillier, Baker-Rumely, and others). In particular a proof of \eqref{eq:equidist mandelbrot} along these lines was obtained, prior to \cite{preper}, 
 by Baker and Hsia  in  \cite[Theorem~8.15]{baker-hsia}).  

\medskip

 Since the varieties  $\percrit(n,0)$ and $\percrit(n,k)$ have generally many irreducible components, 
(e.g. $\percrit(n-k, 0)\subset\percrit(n,k)$), and since it is difficult  to control multiplicities, the theorem 
does not directly imply that $T_c$ is  approximated by parameters where $c$ is genuinely preperiodic.
To ensure this, we use a little trick based on the fact that $T_c$ gives no mass to subvarieties (since it has local continuous potentials).

Denote by  $\prepercrit(n,k)\subset\percrit(n,k)$ be the union of irreducible components of $\percrit(n,k)$ (with their multiplicities) along which $c$ is strictly preperiodic at generic parameters. As sets we have that $$\prepercrit(n,k) =\overline{\percrit(n,k)\setminus \percrit(n-k,0)}.$$

\begin{cor}\label{cor:strpreper}
Under the assumptions of the theorem, if $k$ is fixed, then
$$\unsur{d^n+d^{(1-e)(n-k)}}[\prepercrit({n,n-k})] \cv T_c~.$$
\end{cor}

\begin{proof}
$[\percrit({n,n-k})] - [\prepercrit({n,n-k})]= [D_n]$ is a sequence of effective divisors supported on  $\percrit(k, 0)$. Assume  by contradiction that the conclusion of the corollary does not hold. Then  $T_c$ would give positive  mass to   $\percrit(k, 0)$, which cannot happen.
\end{proof}

Here is a heuristic geometric argument justifying the validity of  Theorem \ref{thm:cvg preper}. For each $\la$, (pre)periodic points equidistribute towards the maximal measure $\mu_\la$ \cite{lyubich}. For this one ``deduces'' that in $\La\times \pu$, the sequence of integration currents on the hypersurfaces 
$$\set{(\la, z)\in \La\times\pu, \ f^n_\la(z) = f^{k(n)}_\la(z)},$$ conveniently normalized,  converge to $\widehat T$. 
By ``restricting" this convergence to the graph  $\widehat c$, one gets the desired result. The   trouble here is that there is  no general result showing that 
 that the slices $\widehat{T}_n\rest{\widehat c}$ converge to $\widehat{T}\rest{\widehat c}$.
 
 \begin{proof}[Proof of Theorem \ref{thm:cvg preper} (sketch)]
 Assume for simplicity that $(f_\la)$ is a family of polynomials of degree $d$. A psh potential of $d^{-n}\percrit({n,k(n)})$ is given by $d^{-n} \log\abs{f^n_\la(c(\la)) - f^{k(n)}_\la(c(\la))}$. We need to show that this sequence  converges to $G_\la(c(\la))$, or equivalently that 
 \begin{equation}\label{eq:cv zero}
 \unsur{d^n} \log\abs{f^n_\la(c(\la)) - f^{k(n)}_\la(c(\la))} - \unsur{d^n}\log^+\abs{f^n_\la(c(\la))} \underset{n\cv\infty}\longrightarrow 0 \text{ in } L^1_{\rm loc}(\La)
 \end{equation}
 converges to zero. We argue by   case by case analysis depending on the behavior of $c$. For instance it is clear that \eqref{eq:cv zero} holds when $c$ escapes to infinity, or is attracted to an attracting cycle. On the other hand there are parts of parameter space 
 where the convergence is delicate to obtain directly. So instead we apply some  potential-theoretic ideas (slightly reminiscent of the proof of   Brolin's theorem \cite{brolin}). 
 One of these arguments is based on the maximum principle, and requires a certain compactness property leading to assumption (H). \end{proof}

\medskip

The speed of convergence in Theorem \ref{thm:cvg preper} is unknown in general. Furthermore the proof   is ultimately based on compactness properties of the space of psh functions, so it is not well suited to obtain such an estimate.
 The only  positive result in this direction is due to Favre and 
 Rivera-Letelier \cite{frl}, based on the method of \cite{baker-hsia}, and concerns the unicritical family.

\begin{thm}[Favre-Rivera Letelier]\label{thm:frl}
Consider  the unicritical family of polynomials $(z^d+\la)_{\la\in \cc}$. Then for any  any compactly
  supported $C^1$ function $\varphi$, if $0\leq k(n)<n$ is any sequence, as $n\cv\infty$   we have   
  $$
\abs{\bra{\percrit(n,k(n)) - \tbif, \varphi}}\leq C  \lrpar{\frac{n }{d^{n}}}^{1/2} \norm{\varphi}_{C^1}.
 $$
\end{thm}

\subsection{Equidistribution of parameters with  periodic orbits of a given multiplier}\label{subs:multiplier}
  Bassanelli and Berteloot studied in \cite{bas-ber2, bas-ber3} the distribution of parameters for which there exists a periodic cycle of a given multiplier. For this, given any holomorphic family of rational maps $(f_\la)$, we need  to  define  the subvariety $\per (n,w)$   of parameter space defined by the condition that $f_\la$ admits a cycle of exact period $n$ and multiplier $w$. Doing this consistently as $w$ crosses the value 1 requires a little bit of care.
 The following result, borrowed from \cite{bas-ber2}, originates from  the work of Milnor \cite{milnor quadratic} and Silverman \cite{silverman book}.

\begin{thm}
 Let $(f_\la)$ be a holomorphic family of rational maps of degree $d\geq 2$. Then for every integer $n$
 there exists a holomorphic function $p_n$ on $\La\times \cc$, which is polynomial on $\cc$,  such that 
\begin{enumerate}
 \item[i.] For any $w\in \cc\setminus\set{1}$, $p_n(\la, w)=0$ if and only if $f_\la$ admits a cycle of exact period $n$ and of multiplier $w$;
\item[ii.] For $w=1$, $p_n(\la, 1)=0$ if and only if $f$ admits a cycle of exact period $n$ and of multiplier 1 or a cycle of exact period $m$ whose multiplier is a primitive $r^{\rm th}$ root of unity, and $n=mr$.
\end{enumerate}
\end{thm}

We  now put $\per (n,w) = \set{\la, \ p_n(\la, w) =0}$. The equidistribution result is the following
\cite{bas-ber2, bas-ber3} (recall our convention that $[\La]=0$).

\begin{thm}[Bassanelli-Berteloot]\label{thm:basber}
Let $(f_\la)_{\la\in \La}$ be a holomorphic family of rational maps of degree  $d\geq 2$. Then
\begin{enumerate}%[i.]
\item[i.] for any  $w\in \cc$ such that $\abs{w}<1$, $\displaystyle \unsur{d^n} \left[\per (n,w)\right]\cv \tbif$.
\item[ii.] Let $d\theta$ denote  the normalized Lebesgue measure on $\re/2\pi\zz$. Then for every $r>0$,  
 $$ \unsur{d^n} \int_{\re/2\pi\zz} \left[\per (n,re^{i\theta})\right] d\theta \cv \tbif;$$
 \end{enumerate}

If moreover  $(f_\la)_{\la\in \La}$ is the family of all polynomials of degree $d$, then
\begin{enumerate}[iii.]
\item for any  $w$ such that $\abs{w}\leq 1$, $\displaystyle \unsur{d^n} \left[\per (n,w)\right]\cv \tbif$.
\end{enumerate}
\end{thm}

\begin{proof}[Proof (excerpt)] %The proof of items {\em i.} and {\em ii.} is rather easy so let us explain the argument. 
For fixed $\la$, the polynomial $w\mapsto p_n(\la,w)$ can be decomposed into
$$p_n(\la, w) = \prod_{i=1}^{N_d(n)} (w-w_j(\la)),$$ where the degree $N_d(n)$ satisfies $N_d(n)\sim \frac{d^n}{n}$ and the $w_j(\la)$ are the multipliers of the periodic cycles of period $n$ of $f_\la$. 
 For simplicity,  in {\em i.} let us only discuss  the case where the multiplier $w$ equals 0. 
We can write 
$$\unsur{d^n}[\per (n,0)] = \unsur{d^n}dd^c\log\abs{p_n(\la, 0)} = dd^c\lrpar{ \unsur{d^n}\sum_{i=1}^{N_d(n)}  \log\abs{ w_j(\la)} }.$$ We see that the potential of $\unsur{d^n}[\per (n,0)]$ is just the average value of the  logarithms of the multipliers of repelling orbits. Now for $n$ large enough, all cycles of exact period $n$ are repelling so from \eqref{eq:bdm} we see\footnote{The additional $\unsur{n}$ in that formula follows from the fact that in  \eqref{eq:bdm} the sum is over periodic points while here we sum over periodic cycles.} that the sequence of potentials converges pointwise to 
$\chi(f_\la)$. Furthermore it is easy to see that this sequence is locally uniformly bounded from above so by the Hartogs lemma  it converges in $L^1_{\rm loc}$. By taking $dd^c$ we see that $\unsur{d^n}[\per (n,0)]$ converges to $\tbif$. 

\medskip

The argument for {\em ii.} is similar. For simplicity assume that $r=1$.  We write
\begin{align*}
\unsur{d^n}\int& \left[\per (n,e^{i\theta})\right] d\theta = 
\unsur{d^n} dd^c \lrpar{ \int \log\abs{p_n(\la, e^{i\theta})} d\theta} \\ &=  dd^c\lrpar{\unsur{d^n} \sum_{i=1}^{N_d(n)} 
\int \log\abs{e^{i\theta} -w_j(\la)} d\theta } = dd^c\lrpar{ \unsur{d^n}
\sum_{i=1}^{N_d(n)} \log^+\abs{w_j(\la)}},
\end{align*} 
where the last equality follows from the well-known formula $\log^+\abs{z} = \int\log\abs{z- e^{i\theta}}d\theta$. As before for each $\la$, when  $n$ is large enough all points of period $n$ are repelling, so the potentials converge pointwise to $\chi$.  We conclude as before. 

\medskip

The proof of {\em iii.} is more involved since for $\abs{w}=1$, in the estimation of the potentials  we have to deal with the possibility of cycles of large period with multipliers close to $w$.  
To overcome this difficulty, the authors use a global argument (somewhat in the spirit of the use of (H) in Theorem \ref{thm:cvg preper}), requiring the additional   assumption that $(f_\la)$ is the family of polynomials.   
\end{proof}

It is a useful fact that in assertions {\em i.} and {\em ii.} of Theorem \ref{thm:basber}, no global assumption on $\La$ is required (see below Theorem 
\ref{thm:tbifk per}). On the other hand in {\em iii}, 
  it is expected that the convergence of $d^{-n}[\per (n,w)]$ to $\tbif$ (even for $\abs{w}>1$) holds in any family of rational mappings. Note that by using techniques similar to those of Theorem \ref{thm:equidist abstrait}, it can be shown that in any family of rational maps, the set of $w\in \cc$ violating the convergence in   {\em iii} is polar.

%%%%%%%%%%%%%%%%%%%%%%%%%%%%%%%%%%%%%%%%%%%%%%

\section{Higher bifurcation currents and the bifurcation measure}\label{sec:higher}

A crucial difference between one and   higher dimensional families of rational mappings is the presence of a hierarchy of bifurcations according to the number of bifurcating critical points. These ``higher bifurcation loci" are rather delicate  to define precisely and the main thesis in this section is that the formalism of bifurcation currents is well suited to deal with these questions. 

\medskip

In this respect, let us start  by suggesting
  a certain ``dictionary'' of   analogies between these issues and the dynamics of holomorphic endomorphisms of projective space $\pp^k$, which turns out to be a very instructive guide for the intuition. Let $f$ be a holomorphic self map of degree $d$ on  $\pk$. There exists a natural invariant   positive closed  current $T$  of bidegree (1,1) satisfying $f^*T = dT$  (the Green's current), whose support is the Julia set of $f$ \cite{fs2}. In dimension 1, the dynamics of $f$ is generically expanding along the Julia set. In higher dimension, the situation is more subtle in that 
that ``the number of directions'' along which the iterates are not equicontinuous can vary from $1$ to $k$. One then introduces the following filtration of the Julia set 
$$  J_1 =J =\supp(T) \supset \cdots \supset J_q =  \supp(T^q) \supset \cdots \supset J_k =  \supp(T^k) = \supp(\mu),$$
where $\mu$ is the unique measure of maximal entropy. 
The dynamics on $J_k$ is   ``repelling in all directions'' according to the work of Briend and Duval \cite{briend-duval1}. On the other hand for $q<k$, the dynamics along $J_q\setminus J_{q+1}$ is expected to be ``Fatou in codimension $q$''. It is not completely obvious how to formalise this precisely (see \cite{fatou} for an account). A popular way to understand this is to conjecture that $J_q\setminus J_{q+1} = \supp(T^q)\setminus    \supp(T^{q+1})$ is filled
(in a measure theoretic sense) with holomorphic disks of codimension $q$ along which the dynamics of $(f^n)$ is equicontinuous. 

In this section we will try to develop a similar picture for parameter spaces of polynomial and rational maps, with deformation disks playing the role of  Fatou disks, and Misiurewicz parameters  replacing repelling periodic points.

 \medskip

For cubic polynomials with marked critical points, it is also possible to draw  a rather  complete dictionary 
 with the dynamics of polynomial automorphisms of $\cd$ (see \cite{cubic}). 

\subsection{Some general results}
In this paragraph, $(f_\la)$ is a general holomorphic family of rational maps of degree $d\geq 2$. Our purpose here is to introduce the higher bifurcation currents $\tbif^k$ and study some of their   properties. We will try to demonstrate that their successive supports $\supp(\tbif^k)$, 
$1\leq k \leq \dim(\La)$ define a dynamically meaningful filtration of the bifurcation locus. 

\medskip

We first observe that  it is harmless to assume that all critical points are marked. Indeed, let us take a branched cover $\pi:\widetilde\La\cv\La$ such that   the new family $\widetilde{f}(\widetilde\la)$ has all critical points marked. We claim that for every $1\leq k\leq \dim(\La)$, 
(with obvious notation) $$\pi^{-1}\supp(\tbif^k) = \supp(\widetilde T_{\rm bif}^k).$$
Indeed the Lyapunov exponent function in $\widetilde\La$ is $\widetilde\la\mapsto \chi(\pi(\widetilde\la))$. In particular in any open subset $U\subset\widetilde{\La}$ where $\pi$ is a biholomorphism, for every $k$, we have that $\widetilde\la\in \supp(\widetilde T_{\rm bif}^k)$ iff $\la=\pi(\widetilde \la)\in\supp(\tbif^k)$. 
Let now $\widetilde B$ denotes the branching locus of $\pi$. Since $\chi$ is continuous,  $\tbif^k$ (resp. $\widetilde T_{\rm bif}^k$) gives no mass to analytic subsets, so we infer  that 
$$\supp(\widetilde T_{\rm bif}^k) = \overline{\supp(\widetilde T_{\rm bif}^k)\setminus \widetilde B} \text{ (resp.  }
 \supp(  \tbif^k) = \overline{\supp(  \tbif^k)\setminus \pi(\widetilde B)} \text{ )}.$$ Thus our claim follows. 

\medskip

Therefore we assume that  critical points are marked as $(c_1(\la), \ldots, c_{2d-2}(\la))$, and
  denote by $T_1, \ldots , T_{2d-2}$ the respective bifurcation currents. Recall that $\supp(T_i)$ is the activity locus of $c_i$ and  $\tbif$   equals $\sum T_i$ by Proposition \ref{prop:sum ci}. 

Since the $T_i$ are (1,1) positive closed currents with local continuous potentials, it is possible to wedge them (see \cite{demailly}).
Here is a first observation.

\begin{prop}
For every $1\leq i\leq 2d-2$, $T_i\wedge T_i =0$. 
\end{prop}

\begin{proof}
Assume that the convergence theorem \ref{thm:cvg preper} holds in $\La$ (e.g. if (H) holds). Then $T_i = \lim_{n\cv\infty} d^{-n} [\percrit_i(n,0)]$ where of course $\percrit_i(n,k)$ is  the subvariety of parameters such that $f^n_\la(c_i(\la)) = f^k_\la(c_i(\la))$. Since $T_i$ has continuous potentials, we  infer that $T_i\wedge T_i = \lim_{n\cv\infty} d^{-n} [\percrit_i(n,0)]\wedge T_i$. Now for every $n$, 
$[\percrit_i(n,0)]\wedge T_i = T_i\rest{\percrit_i(n,0)}$ vanishes since $c_i$ is passive on $\percrit_i(n,0)$, and we conclude that $T_i\wedge T_i=0$.

Without assuming Theorem \ref{thm:cvg preper}, the 
 proof is   more involved and due to Gauthier \cite{gauthier}.
 \end{proof}

\medskip

As a consequence of this proposition, we infer that for every $1\leq k\leq \dim(\La)$
\begin{equation}\label{eq:tbifk}
\tbif^k = k! \sum_{1\leq i_1<\cdots <i_k\leq {2d-2}} T_{i_1}\wedge \cdots \wedge T_{i_k}.
\end{equation}
In particular 
\begin{equation}\label{eq:inclusion}
\supp(\tbif^k)\subset \set{\la, \ k \text{ critical points are active at }\la}.
\end{equation}
As we will see in a moment (Example \ref{ex:douady}  below), this inclusion is in general {\em not} an equality. It thus becomes an interesting problem, still open in general,  to characterize $\supp(\tbif^k)$. 
%Charles n'aime pas trop mais je ne vois pas.
We will try to develop the idea that  $\supp(\tbif^k)$ is the set of parameters where $k$ critical points are active and ``behave independently". 

The following result follows from Theorem \ref{thm:basber}.  

\begin{thm}\label{thm:tbifk per}
Let $(f_\la)_{\la\in \La}$ be a holomorphic family of rational maps of degree $d\geq 2$. 
Then for every $k\leq \dim(\La)$, 
$$\supp(\tbif^k)\subset \overline{\set{\la,\ f_\la \text{ admits }k\text{ periodic critical points}}}.$$
\end{thm}

\begin{proof}
We argue by decreasing induction on $k$, by using the following principle: if $V\subset\La$ is a  smooth analytic hypersurface  and $T$ a positive closed (1,1) current with continuous potential, then 
$T^k\wedge [V]= (T\rest{V})^k$. Here as usual the current $T\rest{V}$ is defined by restricting the potential of $T$ to $V$ and taking $dd^c$.

Under the assumptions of the theorem,  let $\la_0\in \supp(\tbif^k)$. Since $$\tbif = \lim_{n\cv\infty} \unsur{d^n}[\per (n,0)]$$ and $\tbif$ has continuous potential,  we infer that 
$$\tbif^k = \lim_{n\cv\infty} \unsur{d^n} [\per (n,0)]\wedge \tbif^{k-1}.$$ In particular, $\la_0$ is approximated by parameters belonging to $\lrpar{\tbif\rest{\per (n,0)}}^{k-1}$ (moving slightly if necessary we can always assume that these belong to the smooth part of $\per (n,0)$). 

We can now put $\La_1= \per (n,0)$ and repeat the argument to find a nearby parameter belonging to $\supp [\per({n_1},0)]\wedge \tbif^{k-2}\subset  \La_1$ for some (possibly much  larger) $n_1$,  etc.\end{proof}

We see that it is important in this argument that no special assumption on $\La$ is needed in Theorem \ref{thm:basber}. For instance if one were to replace ``periodic" by ``strictly preperiodic" in this theorem, and try  to use Theorem \ref{thm:cvg preper}, one would have to check the validity of (H) in the restricted submanifolds, which needn't be satisfied (see however Theorem  \ref{thm:tbifk preper} below). 

 \medskip

We can now explain how the inclusion in \eqref{eq:inclusion} can be strict. 

\begin{exam}[Douady] \label{ex:douady}
 In the two dimensional space of cubic polynomials with marked critical points, let 
$P_0(z) = z+\frac{z^2}{2} + z^3$. We claim that the two critical points are active at $P_0$, but $P_0$ does not belong to $\supp(\tbif^2)$. 

Indeed since $P_0$ is real and the critical points are not real, by symmetry, both critical points are attracted to the parabolic fixed point at the origin.  Since  the fixed point can be perturbed to become repelling (hence does not attract any critical point), Theorem \ref{thm:mss}{\em v.} implies that at least one critical point must be active. Hence by symmetry again, both critical points are active. On the other hand it can be proven (see \cite[Example 6.13]{preper} for details) that any nearby parameter admits an attracting (and not superattracting) fixed point. In particular  $P_0$ cannot be perturbed to make both critical points   periodic, therefore $P_0\notin \supp(\tbif^2)$. 

Denote by $(P_{\la})_{\la\in \La\simeq\cd}$ the  family of cubic polynomials  with marked critical points 
  $c_1$ and $c_2$. We see that if $N$ is a small neighborhood of the above parameter 0, the values of $(P_{\la}^n(c_1(\la)), P_{\la}^n(c_2(\la)))$ for $\la\in N$   avoid an open set in $\cd$. Indeed for $\la\in N$, either $c_1$ or $c_2$ must be attracted by an attracting cycle, so for large $n$,  $P_{\la}^n(c_1(\la))$ and $P_{\la}^n(c_2(\la))$ cannot be simultaneously large. This is a manifestation of the Fatou-Bieberbach phenomenon (failure of Montel's theorem in higher dimension). 
\qed
\end{exam}

One can also approximate $\supp(\tbif^k)$ by parameters possessing  
$k$ indifferent periodic cycles.
For $N_k=(n_1, \ldots n_k)\in \nn^k$ and $\Theta_k=(\theta_1,\cdots ,\theta_k) \in (\re/\zz)^k$, we denote by $\per_k(N_k, e^{i\Theta_k})$ the union of codimension $k$ irreducible components of 
$$\per(n_1, e^{i\theta_1})\cap \per(n_2, e^{i\theta_2})\cap \cdots \cap\per(n_k, e^{i\theta_k}), $$
and if $E\subset \re/\zz$, we let $$\mathcal{Z}_k(E) = \bigcup_{N_k\in \nn^k, \Theta_k\in E^k} \per_k(N_k, e^{i\Theta_k}), $$ which is the (codimension $p$ part of the) set of parameters possessing $k$ neutral cycles with respective multipliers in $E$.

The following result is due to Bassanelli and Berteloot \cite{bas-ber1}.

\begin{thm}[Bassanelli-Berteloot]\label{thm:indifferent}
Let $(f_\la)_{\la\in \La}$ 
be  a holomorphic family of rational maps of degree $d\geq 2$. If $E\subset \re/\zz$ is any dense subset,  then for every $k\leq \dim(\La)$, 
$$\supp(\tbif^k)\subset \overline{\mathcal{Z}_k(E)}.
$$
\end{thm}

\begin{proof}[Proof (sketch)] We argue by induction on $k$. For $k=1$, the result follows from  Theorem \ref{thm:mss}.   The main idea of the induction step is as follows: assume $\la_0\in \supp(\tbif^k)$. Then by the induction hypothesis, $\la\in \overline{\mathcal{Z}_{k-1}(E)}$, so there are plenty of $(k-1)$-dimensional disks near $\la_0$ along which $f_\la$ possesses $(k-1)$ neutral cycles with multipliers in $E$. Assume that the dynamics is $J$-stable along these disks. Then the Lyapunov exponent function $\chi$ is pluriharmonic there and it follows from a general pluripotential theoretic lemma (see \cite[Lemma 6.10]{fs2}) that $(dd^c\chi)^k=0$ in the neighborhood of $\la_0$, a contradiction. So the dynamics is not $J$-stable along the disks of $\mathcal{Z}_{k-1}(E)$, hence 
  Theorem \ref{thm:mss}   produces one more neutral cycle, thereby proving the result.  
 \end{proof}

We see that in    Example  \ref{ex:douady},   
the two critical points are related in a rather subtle way.  On the other hand, under a certain transversality assumption,  a clever argument of similarity between parameter and dynamical spaces, due to      Buff and Epstein,  shows  that certain parameters belong to 
$\supp(\tbif^k)$ \cite{buff epstein}. 

Let $\la_0\in \La$ be a parameter where $c_1(\la_0), \ldots , c_k(\la_0)$ fall onto repelling cycles. More precisely we assume that there exist repelling periodic  
points $p_1(\la_0), \ldots p_k(\la_0)$ and integers $n_1,\ldots, n_k$ such that for $1\leq j\leq k$, $f^{n_j}_{\la_0}(c_j(\la_0)) = p_j(\la_0)$. The repelling orbits $p_j(\la_0)$ can be uniquely continued to repelling periodic orbits $p_j(\la)$ for $\la$ in some neighborhood of $\la_0$. 
Fix for each $j$ a coordinate chart  on $\pu$  containing $p_j(\la_0)$, so that for nearby $\la$, the function $\chi_j : \la\mapsto f^{n_j}_{\la_0}(c_j(\la )) - p_j(\la )$ is well defined. We say that the critical points $c_j(\la_0)$ {\em fall transversely} onto the  respective repelling points  $p_j(\la_0)$ if the mapping $\chi: \La\cv\cc^k$ defined in the neighborhood of $\la_0$ by $\chi =(\chi_1, \ldots , \chi_k)$ has rank $k$ at $\la_0$. Of course this notion does not depend on the choice of coordinate charts. 

 \begin{thm}[Buff-Epstein]\label{thm:buff epstein} 
Let $(f_\la)_{\la\in \La}$ 
be  a holomorphic family of rational maps of degree $d\geq 2$ with marked critical points
$c_1,\ldots, c_k$, and associated bifurcation currents $T_1,\ldots , T_k$. 
Let $\la_0$ be a parameter at which  $c_1(\la_0), \ldots , c_k(\la_0)$ fall transversely onto repelling cycles. Then $\la_0\in \supp(T_1\wedge \cdots \wedge T_k)$. 
\end{thm}

%This result was recently extended by Gauthier \cite{gauthier} in two ways: 
%\begin{itemize}
% \itm the transversality assumption is relaxed to a dimensional assumption, namely, it is only required that $\chi^{-1}(0)$ has codimension $k$;
%\itm  the critical points $c_1(\la_0), \ldots , c_k(\la_0)$ can now fall into an arbitrary hyperbolic set. 
%\end{itemize}
Notice that this theorem does not appear in this form in \cite{buff epstein}. Our presentation borrows from %\cite{gauthier} and 
\cite{berteloot survey}. The validity of the transversality assumption will be discussed in various situations in \S \ref{subs:poly} and \ref{subs:rat} below.

\begin{proof}[Proof (sketch)]
First, a  slicing argument shows that it is enough to prove that for a generic $k$-dimensional subspace $\La'\ni\la_0$ (relative to some coordinate chart in $\La$), the result holds by restricting to $\La'$. Thus we can assume that $\dim(\La)=k$ and that $\chi$ is a local biholomorphism on $\La$. 
To simplify notation,  we will assume that the $f_j$ are polynomials, $k=2$, $n_j=1$ and that the $p_j$ are fixed points.  

Taking adapted coordinates $(\la_1, \la_2)$ in $\La$ (in which the initial parameter $\la_0$ is 0) we can assume that $\chi_1(\la)  =  \la_1+ {h.o.t.}$ and $\chi_2(\la)  =  \la_2+ {h.o.t.}$
The proof, based on a renormalization argument, consists in  estimating the mass, relative to the measure 
 $T_1\wedge T_2$, of   small bidisks about 0  of   carefully chosed  size.   Specifically, we will show  that
$$\liminf d^n (T_1\wedge T_2) \lrpar{D \lrpar{0, \frac{\delta}{m_1^n} }\times  D \lrpar{0, \frac{\delta}{m_2^n} }} >0,$$
where the $m_j$ are the respective multipliers of the $p_j$. Let $\delta_n$ be the scaling map defined by 
$$\delta_n(\la) =\left(\frac{\la_1}{m_1^n},  \frac{\la_2}{m_2^n}\right).$$
An easy computation based on transversality and  the fact that $f_\la$ is linearizable near $p_j(\la)$ shows that for $j=1,2$, $f_{\delta_n(\la)}^{n+1}(c_j(\delta_n(\la)))$ converges 
as $n\cv\infty$ to a non-constant map $\psi_j$  {\em depending only on} $\la_j$, on some disk $D(0, \delta)$, with $\psi_j(0) =p_j$. Hence, since $G_\la$ depends continuously on $\la$ we get that $G_{\delta_n(\la)}(f_{\delta_n(\la)}^{n+1}(c_j(\delta_n(\la))))$ converges to $G_0\circ \psi_j(\la_j)$. Using the invariance relation for the Green's function, we conclude that  
\begin{equation}\label{eq:delta n}
d^{n+1} G_{\delta_n(\la)}( c_j(\delta_n(\la))) = 
G_{\delta_n(\la)}\left(f_{\delta_n(\la)}^{n+1}(c_j(\delta_n(\la)))\right) \underset{n\cv\infty}\longrightarrow
G_0(\psi_j(\la_j)), 
\end{equation}  
hence 
\begin{align*}&(T_1\wedge T_2)(\delta_n (D(0, \delta)^2)) \simeq d^{-2(n+1)} \int_{D(0, \delta)^2} dd^c\left(G_0\circ \psi_1 (\la_1)\right)\wedge dd^c \left(G_0\circ \psi_2(\la_2)\right) \\ &= d^{-2(n+1)}
\lrpar{\int_{D(0, \delta)} dd^c\left(G_0\circ \psi_1 (\la_1)\right) }\lrpar{\int_{D(0, \delta)} dd^c\left(G_0\circ \psi_2 (\la_2)\right) } \text{ by Fubini's theorem}.
\end{align*}
Note that the first line of this equation 
is justified by the local uniform convergence in \eqref{eq:delta n}.
 Finally, the integrals on the second line are positive since $G_0$ is not harmonic near $p_j$, so $G_0\circ \psi_j$ is not harmonic near the origin. 
\end{proof}

Building on similar ideas, Gauthier \cite{gauthier} relaxed 
the transversality assumption in Theorem \ref{thm:buff epstein}
as follows. Assume as before that $\la_0\in \La$ is a parameter where $c_1(\la_0), \ldots , c_k(\la_0)$ fall onto respective repelling   periodic  
points $p_1(\la_0), \ldots p_k(\la_0)$. Define $\chi:\La\cv \cc^k$ as before Theorem \ref{thm:buff epstein}. We say that  the critical points $c_j(\la_0)$ {\em fall properly} onto the  respective repelling points  $p_j(\la_0)$ if $\chi^{-1}(0)$ has codimension $k$ at $\la_0$. To say it differently, we are requesting that in  $\La\times (\pu)^k$ the     graphs of the two mappings 
$\la\mapsto (p_1(\la), \ldots p_k(\la))$ and $\la\mapsto (f^{n_1}_\la(c_1(\la)), \ldots , 
f^{n_k}_\la(c_k(\la)))$ intersect properly\footnote{The   is the usual  terminology 
   in intersection theory, 
   see \cite{fulton, chirka}.} at $(\la_0, p_1(\la_0), \ldots p_k(\la_0))$.

\begin{thm}[Gauthier]\label{thm:gauthier proper} 
Let $(f_\la)_{\la\in \La}$ 
be  a holomorphic family of rational maps of degree $d\geq 2$ with marked critical points
$c_1,\ldots, c_k$, and associated bifurcation currents $T_1,\ldots , T_k$. 
Let $\la_0$ be a parameter at which  $c_1(\la_0), \ldots , c_k(\la_0)$ fall properly onto repelling cycles. Then $\la_0\in \supp(T_1\wedge \cdots \wedge T_k)$. 
\end{thm}

Actually, it is enough that $c_1(\la_0), \ldots , c_k(\la_0)$ fall properly into an arbitrary 
 hyperbolic set. We refer to \cite{gauthier} for details.
 
 \medskip

The results in this section show that for $1\leq k\leq \dim(\La)$, $\supp(\tbif^k)$ is a reasonable candidate for the locus of ``bifurcations of order $k$". We will see in the next sections that when $\La$ is the space of all polynomials or rational maps and $k$ is maximal,    $\supp(\tbif^k)$ can be characterized precisely.
 The picture is not yet complete in intermediate codimensions. In this respect let us state a few open questions. 
 
 \begin{question} \label{q:tbifk}
 Let $(f_\la)_{\la\in \La}$ be a holomorphic family of rational maps on $\pu$ 
 of degree $d\geq 2$, with marked critical points $c_1, \ldots, c_k$.
 \begin{enumerate}[1.]
 \item Is it true that 
 $$ \supp(T_1\wedge \cdots \wedge T_k) = \overline{\set{\la_0, \ c_1(\la_0), \cdots ,c_k(\la_0) \text{ fall transversely onto   repelling cycles}}} $$ (by Theorem \ref{thm:buff epstein} only the inclusion $\subset$ needs to be established).
 %\item  Is it further possible to arrange that 
%$f^n(c_1(\la_0)), \cdots, f^n(c_k(\la_0)) $ belong to repelling cycles for the same value of $n$?
\item More generally, do the codimension $k$ subvarieties   $$\percrit_1(n, k(n))\cap \cdots \cap  \percrit_k(n, k(n))$$ equidistribute towards $T_1\wedge \cdots \wedge T_k$?    Arithmetic methods could help here, especially when $\La$ is the space of polynomials and $k$ is  maximal (see the next paragraph). 
\item Is the following characterization of $\supp(T_1\wedge \cdots \wedge T_k)$ true: $\la_0\in \supp(T_1\wedge \cdots \wedge T_k)$ iff for every neighborhood $U$ of $\la_0$, there exists a pluripolar subset $\mathcal{E}\subset(\pu)^k$ such that the values of $f^n_\la(c_j(\la))$ for   $n\in \nn$ 
and $\la\in U$ cover $ (\pu)^k\setminus \mathcal{E}$?
\end{enumerate}
\end{question}
  
Work in progress of the author indicates that the answer to Question \ref{q:tbifk}.{1.} should be ``yes''.  

%finir la section sur les caractérisations conjecturales penser aux exemples stupides genre z^2+c_1+c_2$

\subsection{The space of polynomials}\label{subs:poly}
In this paragraph we specialize the discussion to the case where $\La$ is the space of polynomials of degree $d$, and will mostly concentrate on  the maximal exterior power of the bifurcation current (the {\em bifurcation measure}). Our purpose is to show that  it is in many respected 
the right analogue in higher degree of the   harmonic measure of the Mandelbrot set. All results except Theorem 
\ref{thm:HD polynomial} come  from \cite{preper}. 

\medskip

The space $\mathcal{P}_d$ of polynomials of degree $d$ with marked critical points is a singular affine algebraic variety. To work on this space, in practice we use an ``orbifold parameterization" (not injective) $\pi:\cc^{d-1}\cv\mathcal{P}
_d$, defined as follows: $\pi$ maps $(c_1, \cdots, c_{d-2},a)
\in\mathbb{C}^{d-1}$ to the primitive of $z\prod_1^{d-2} (z-c_i)$
whose value at $0$ is $a^d$. In coordinates, denoting    $c=(c_1, \ldots , c_{d-2})$, we get 
\begin{equation}\label{e:111}
\pi(c,a) =  {P}_{c,a} (z) 
= \frac1d\, z^d+\sum_{j=2}^{d-1}(-1)^{d-j}\, \sigma_{d-j}(c)\,
\frac{z^j}{j} + a^d~,  
\end{equation}
where $\sigma_i(c)$ is the elementary symmetric polynomial in the  $\{c_j\}_1^{d-2}$
 of degree $i$.  The critical points of $P_{c,a}$ are $\{0, c_1, \cdots , c_{d-2}\}$. We put $c_0 = 0$.

The choice of this  parameterization (inspired from that used by Branner and Hubbard in \cite{branner-hubbard1}) is motivated by the fact that the bifurcation currents $T_i$ associated to the $c_i$ have the same projective mass. Furthermore, it is well suited in order  to understand the behavior at infinity 
 of certain parameter space subsets (this will be used to check condition (H) of Theorem 
 \ref{thm:cvg preper}).  
 
We let $\mathcal{C}$ be the connectedness locus, which is compact in $\cc^{d-1}$ by \cite{branner-hubbard1}. For $1\leq i\leq d-2$ we also define the closed subsets $\mathcal{C}_i$ by  
$$\mathcal{C}_i 
= \set{(c,a), \  c_i\text{ has bounded orbit}}.$$ It is clear that  $\mathcal{C} = \bigcap_i\mathcal
{C}_i$ and that  $\fr \mathcal{C}_i$ is the activity locus of  $c_i$. 

We  also let $g_{c,a}$ be the Green's function of the polynomial $P_{c,a}$. Then $T_i = dd^cg_i$, where $g_i = g_{c,a}(c_i)$. Recall that the Manning-Przytycki formula asserts that 
$\chi = \log d+ \sum_{i=0}^{d-2} g_i$, hence $\tbif = \sum_{i=0}^{d-2} T_i$. 

\medskip

In this specific  situation we are able to solve the problem raised after Theorem \ref{thm:tbifk per}. 

\begin{thm}\label{thm:tbifk preper}
In $\mathcal{P}_d$, for every $1\leq k\leq d-1$, 
$$\supp(\tbif^k)\subset \overline{\set{\la,\ f_\la \text{ admits }k\text{ strictly 
preperiodic critical points}}}.$$

More precisely, for any collection of integers $i_1, \cdots , i_k \in \{ 0,
  \cdots , d-2\}$, the analytic subset  $W_{n_1,  \ldots, n_k} =
  \bigcap_{j=1}^k\prepercrit_{i_j}(n_j,n_j-1)$ is of pure codimension $k$ and 
  \begin{equation}\label{eq:induc}
\lim_{n_k \cv \infty} \cdots \lim_{n_1\cv\infty}   \,
\frac1{d^{n_k+ \cdots +n_1 }}\, [W_{n_1,
  \ldots, n_k}] =T_{i_1} \wedge \cdots \wedge
T_{i_k}~.
\end{equation}
\end{thm}

Of course in \eqref{eq:induc} we can replace $\prepercrit(n,n-1)$ by $\percrit(n,k(n))$ for any sequence $k(n)$.
As already observed, to prove this theorem one cannot simply take wedge products in Corollary \ref{cor:strpreper} (resp. Theorem \ref{thm:cvg preper}). The proof goes by induction on $\ell\leq k$, by successively applying this corollary to the parameter space $\La = W_{n_1, \cdots, n_\ell}$. It is not obvious to check that assumption (H) is satisfied, and at this point  the particular choice of the parameterization is useful (see also \cite{berteloot survey} for neat computations).  It is unknown whether in \eqref{eq:induc} one can take $n_1= \cdots = n_k=n$ (see Question \ref{q:tbifk}.2.)

\medskip

Let us now focus on the maximal codimension case $k = d-1$. We set  $\mu_{\rm bif} = 
T_0\wedge \cdots \wedge T_{d-2} = \unsur{(d-1)!}\tbif^{d-1}$, which is a probability measure supported on the boundary of the connectedness locus.  

\begin{prop}\label{prop:equilibrium}
The bifurcation measure is the pluripotential equilibrium measure of the connectedness locus. In particular  $\supp(\mu_{\rm bif})$ is the Shiloff boundary of $\mathcal{C}$.
\end{prop}

It is important to understand  that when $d\geq 3$,  $\supp(\mubif)$ is a proper subset of $\fr\C$. To get a (crude but instructive!) mental picture of the situation, think about the boundary of a polydisk in $\cc^{d-1}$. This boundary  can be decomposed into foliated pieces of varying dimension between $1$ and $d-2$, together with  the unit torus, 
%$$\fr\dd^{d-1} =\bigcup_{j=0}^{d-1} \dd^j\times \mathbb{T}^{d-1-j},$$
the unit torus $\mathbb{T}^{d-1}$ being the Shiloff boundary. The structure of $\fr\C$ should be somehow similar to this, with foliated pieces of dimension $j$ corresponding to parts of $\fr\C$ where 
$j$   critical points are passive.  The precise picture is far from being understood, except for cubic polynomials (see below \S \ref{subs:laminarity}).

\medskip

It is a well-known open question whether in higher dimension   the connectedness locus is the closure of its interior. Theorem \ref{thm:tbifk per} provides a partial answer to this question: if $\la_0\in \fr \C\cap \supp(\tbif^j)$ and $j$ critical points are active at $\la_0$, then $\la_0\in \overline{\Int (\C)}$. Indeed, there exists a neighborhood $U\ni\la_0$ where $d-1-j$ critical points are passive, hence persistently do not escape, and by Theorem \ref{thm:tbifk per} there is a sequence of parameters 
$\la_n\cv \la_0$ for which the $j$ remaining critical points are periodic. Hence  $\la_n\in  \Int (\C)$ and we are done. 

We see that the answer to the problem lies in the set of parameters in $\fr\C\cap \supp(\tbif^j)$ with more than $j$ active critical points (like in  Example \ref{ex:douady}). So far there does not seem to be any reasonable (even conjectural) understanding of the structure of this set of  parameters.

\medskip

We can give a satisfactory dynamical characterization of $\supp(\mubif)$.
A polynomial is said to be {\em Misiurewicz}   if all critical points fall onto repelling cycles. 

\begin{thm} \label{thm:mis}
$\supp(\mu_{\rm bif})$  is the closure of the set of Misiurewicz parameters. 
\end{thm}

The fact that $\supp(\mubif)$ is contained in the closure of Misiurewicz polynomials follows from Theorem \ref{thm:tbifk preper}, since a strictly postcritically finite polynomial is automatically Misiurewicz. So the point here is to prove the converse. There are actually several 
   proofs of this. The original one  in \cite{preper} uses landing of  external rays (see below).   
%while the other, due to Buff and Epstein \cite{buff epstein} consists in showing that the transversality 
%hypothesis in    Theorem \ref{thm:buff epstein} is satisfied at all Misiurewicz polynomials (see the 
%next paragraph). 
Another proof goes by observing that the properness assumption of Theorem 
\ref{thm:gauthier proper} is satisfied at every Misiurewicz parameter. Indeed, as already  observed in Theorem \ref{thm:tbifk preper}, for every $(n_j)_{j=1, \ldots , d-1}$ and  $(n_j)_{j=1, \ldots , d-1}$ with $k_j<n_j$, 
 $\bigcap_{j=1}^{d-1} \prepercrit_j(n_j, k_j)$ is of dimension 0. Indeed, otherwise,  this intersection would contain an analytic set contained in the connectedness locus, contradicting its compactness. Therefore, at a Misiurewicz point, critical points fall properly on the corresponding repelling points, and Theorem \ref{thm:gauthier proper} applies. 

Notice that the work of Buff and Epstein \cite{buff epstein} implies that these intersections are actually transverse. 

\medskip

  Theorem \ref{thm:gauthier proper} in its general form \cite{gauthier} 
  shows that one can alternately characterize the support of  
$\mubif$ as  being the closure of  {\em generalized Misiurewicz} polynomials, where generalized here means that critical points fall into a hyperbolic set disjoint from the critical set. These considerations lead to a   generalization in higher degree   the  well-known theorem of Shishikura on the Hausdorff dimension of the boundary of the Mandelbrot set \cite{shishikura}. Notice that Tan Lei \cite{tan lei} extended Shishikura's theorem by showing that the bifurcation locus has maximal Hausdorff dimension in any family of rational maps. 

\begin{thm}[Gauthier]\label{thm:HD polynomial}
If $U$ is any open set   such that $U\cap \supp(\mubif)\neq \emptyset$, then 
 $\supp(\mubif)\cap U$ has maximal Hausdorff dimension $2(d-1)$.
  \end{thm}
 
 We now discuss external rays, following \cite{preper}. Let $P\in \mathcal{P}_d$ be a polynomial for which the  Green's function takes the same value $r>0$ at all critical points ($J(P)$ is then a Cantor set). The set
 $\Theta$  of external arguments of external rays landing at the critical points 
 enables to describe in a natural fashion the combinatorics of $P$. This set of angles is known as the {\em critical portrait} of $P$. Now we can deform $P$ by keeping $\Theta$ constant and let $r$ vary in $\mathbb{R}_+^*$ (this is the ``stretching" operation of \cite{branner-hubbard1}). This defines a ray in parameter space corresponding to the critical portrait $\Theta$.
 
The set  $\mathrm{Cb}$ of combinatorics/critical portraits is endowed with a natural Lebesgue measure  $
\mu_{\mathrm{Cb}}$  coming from $\re/\zz$.  It is easy to show that  $\mu_
{\mathrm{Cb}}$-a.e. ray lands as $r\cv 0$ (this follows from Fatou's theorem on the existence of radial limits of bounded holomorphic functions). We thus obtain a measurable landing map
 $e: \mathrm{Cb} \cv \mathcal{C}$.  The measures $\mu_{\rm Cb}$ and $\mu_{\rm bif}$ are related as follows:
 
\begin{thm}\label{thm:landing}  $\mubif$ is the landing measure, i.e. 
$e_*\mu_{\mathrm{Cb}} = \mu_{\rm bif}$.
\end{thm} 

The proof of Theorem \ref{thm:mis} given in \cite{preper} relies on a more precise landing theorem for ``Misiurewicz combinatorics''. A critical portrait $\Theta\in \mathrm{Cb}$ is said to be of Misiurewicz type if the   external angles it contains are strictly preperiodic under multiplication by $d$. A combination of  results 
due to Bielefeld, Fisher and Hubbard \cite{bfh}  and Kiwi \cite{kiwi-portrait} asserts that the landing map $e$ is continuous at  Misiurewicz combinatorics and   that the landing point is a Misiurewicz point. It then follows from Theorem  \ref{thm:landing} that Misiurewicz points belong to $\supp(\mubif)$. 

\medskip

The description of  $\mu_{\rm bif}$ in terms of external rays allows to generalize to higher dimensions a 
result  of Graczyk-\'Swi\c{a}tek \cite{gs} and Smirnov \cite{smirnov}.

\begin{thm}\label{thm:ce}
  The Topological Collet-Eckmann property holds for  $\mu_{\rm
    bif}$-almost every polynomial $P$.
    
 In particular for  a $\mu_{\rm bif}$-a.e. $P$, we have that:
\begin{itemize}
  \itm all cycles are repelling; \itm the orbit of each critical
  point is dense in the Julia set; \itm $K_P=J_P$ is locally connected
  and its Hausdorff dimension is  smaller  than 2.
\end{itemize}
\end{thm}

Notice that Gauthier \cite{gauthier} shows that for a topologically generic polynomial $P\in \supp(\mubif)$, $\mathrm{HD}(J_P)=2$. So the topologically and metrically generic pictures differ. 

\medskip

The connectedness of  $\mathrm{Cb}$ 
naturally suggests the following generalization of the connectedness of the boundary of the  Mandelbrot set.
 
\begin{question} 
Is $\supp(\mubif)$ connected?
\end{question}

\subsection{The space of rational maps}\label{subs:rat}
The space $\mathrm{Rat}_d$ of rational maps of degree $d$ is a smooth complex manifold of dimension $2d+1$, actually a Zariski open set of $\mathbb{P}^{2d+1}$. 
For $d=2$ it is isomorphic to $\cd$ \cite{milnor quadratic}, and it was recently proven  
to be rational for all $d$ by Levy \cite{levy}. Automorphisms of $\pu$ act by conjugation of rational maps, and the moduli space $\mathcal{M}_d$ is the quotient $\mathrm{Rat}_d/\PSL$. Even if the action is not free it can be shown that 
$\mathcal{M}_d$ is a normal quasiprojective variety of dimension $2(d-1)$ \cite{silverman book}. 
Exactly as we did for polynomials, to work on this space, it is usually convenient to work on a smooth family $\La$ which is transverse to the fibers of the projection $\mathrm{Rat}_d \cv \mathcal{M}_d$. 
Through every point of $\mathrm{Rat}_d$ there exists such a family. 
Working in such a family, we can  consider
 elements of $\mathcal{M}_d$ as rational maps rather than conjugacy classes.

\medskip 

In this paragraph we briefly give some properties of the bifurcation measure $\mubif= \tbif^{2(d-1)}$ on $\mathcal{M}_d$. 
The first basic result was obtained in \cite{bas-ber1}. 

\begin{prop}
 The bifurcation measure has  positive and finite mass on $\mathcal{M}_d$. 
\end{prop}

These authors  also give a  nice argument showing that all isolated Lattès examples belong to $\supp(\mubif)$: it is known that a rational map is a Lattès example if and only if its Lyapunov exponent is minimal, that is, equal to $\log d/2$ \cite{ledrappier cras, zdunik}. On the other hand if $u$ is a continuous psh function on a complex manifold of dimension $k$ with a local minimum at $x_0$, then $x_0\in \supp((dd^cu)^k)$ (this follows from the so-called comparison principle  \cite{bt}). The result follows.

\medskip

The precise characterization of $\supp(\mubif)$ is due to Buff and Epstein \cite{buff epstein} and Buff and Gauthier \cite{buff gauthier}. Let
\begin{itemize}
%\itm $\mathcal{X}_d$ be the set of (conjugacy classes of) strictly post-critically finite rational maps, which are  not flexible Lattès examples\footnote{ A {\em flexible Lattès map} is a    rational mapping  descending from an integer multiplication on  elliptic curve. These can be deformed with the elliptic  curve, hence the terminology.};
\itm $\mathrm{SPCF}_d$ be the set of (conjugacy classes of)  strictly post-critically finite rational maps;
 \itm $\mathcal{Z}_d$ be the set of rational maps possessing $2d-2$  indifferent cycles, without counting multiplicities (this is the maximal possible number).
\end{itemize}

\begin{thm}[Buff-Epstein, -Gauthier]
 In $\mathcal{M}_d$, $\supp(\mubif) =  \overline{\mathrm{SPCF}_d} = \overline{\mathcal{Z}_d}$.
\end{thm}

%It is an open question whether  flexible Lattès examples belong  to $\supp(\mubif)$.\note{changer} 

%\medskip

\begin{proof}[Idea of proof] 
It is no loss of generality to assume that all critical points are marked as $\set{c_1,\ldots , c_{2d-2}}$.
Let $\mathrm{SPCF}^*_d$ be the set of (conjugacy classes of) strictly post-critically finite rational maps, which are  not flexible Lattès examples\footnote{ A {\em flexible Lattès map} is a    rational mapping  descending from an integer multiplication on  elliptic curve. These can be deformed with the elliptic  curve, hence the terminology.}.
We already know from Theorem \ref{thm:indifferent} that $\supp(\mubif)\subset \overline{\mathcal{Z}_d}$. It should be possible to prove that $\supp(\mubif)\subset \overline{\mathrm{SPCF}_d}$ in the spirit of Theorem \ref{thm:tbifk preper}, but again it is not easy to check assumption (H). Instead Buff and Epstein prove directly that $\overline{\mathrm{SPCF}^*_d} = \overline{\mathcal{Z}_d}$ 
using the Ma\~né-Sad-Sullivan theorem. 

Let us  show that $\mathrm{SPCF}^*_d \subset\supp(\mubif)$.
 If $f\in \mathrm{SPCF}^*_d$, it satisfies a relation of the form $f^{n_j+p_j}(c_j) = f^{n_j}(c_j)$, $j=1,\ldots ,2d-2$. Consider the subvariety of $\mathcal{M}_d$ defined by these algebraic equations.  
We claim  that this subvariety    
is of dimension 0 at $f$. Indeed a theorem of McMullen \cite{mcm algo} asserts that any stable algebraic  family of rational maps is a family of flexible Lattès maps, hence the result. 
It then follows from Theorem \ref{thm:gauthier proper} that $f\in \supp(\mubif)$. 

The original argument of \cite{buff epstein} was (under some mild restrictions on $f$) to check  
the transversality assumption of Theorem 
\ref{thm:buff epstein}, using Teichmüller-theoretic techniques.  
Another proof of this fact was given by Van Strien \cite{van strien}.

%
%A word about the proof: we already know from Theorem \ref{thm:indifferent} that $\supp(\mubif)
%\subset \overline{\mathcal{Z}_d}$. It should be possible to prove that $\supp(\mubif)\subset \overline
%{\mathcal{X}_d}$ in the spirit of Theorem \ref{thm:tbifk preper}, but again it is not easy to check 
%assumption (H). Instead Buff and Epstein prove directly that $\overline{\mathcal{X}_d} = \overline
%{\mathcal{Z}_d}$ using the Ma\~né-Sad-Sullivan theorem. 
%To show that $\mathcal{X}_d \subset\supp(\mubif)$, these authors prove that the the transversality 
%assumption of Theorem 
%\ref{thm:buff epstein} is satisfied, using Teichmüller-theoretic techniques.  Another proof of this fact 
%was given by Van Strien \cite{van strien}. 

Finally,  by using an explicit deformation,  Buff and Gauthier \cite{buff gauthier} recently proved  that every flexible Lattès map can be approximated by  strictly post-critically finite rational maps which are not Lattès examples. Therefore   $\overline{\mathrm{SPCF}^*_d} = \overline{\mathrm{SPCF}_d}$. 
\end{proof}
  
As in the  polynomial case, in the previous result one may relax the critical finiteness assumption by only requiring that the critical points map to a hyperbolic set disjoint from the critical set.  Similarly to Theorem \ref{thm:HD polynomial} one thus gets \cite{gauthier}:

\begin{thm}[Gauthier]
 The Hausdorff dimension of $\supp(\mubif)$ equals $2(2d-2)$.
\end{thm}

A famous theorem of Rees \cite{rees} asserts that the set of rational maps that are ergodic with respect to Lebesgue measure 
is of positive measure in parameter space. It is then natural to ask: are these parameters inside  $\supp(\mubif)$?  

\subsection{Laminarity}\label{subs:laminarity} 
 A positive current $T$ of bidegree $(q,q)$ in a complex manifold $M$ of dimension $k$ is said to be 
\textit{locally uniformly laminar}
 if in the neighborhood of every point of $\supp(T)$ there exists  a lamination by $q$ dimensional disks embedded in $M$ such that $T$ is an average of integration currents over the leaves. More precisely, the restriction of $T$ to a flow box of this lamination is of the form $\int_\tau [\Delta_t] dm(t)$, where $\tau$ is a local transversal to the lamination, $m$ is a positive measure on $\tau$, and $\Delta_t$ is the plaque through $t$.
 
A current  $T$ of bidegree $(q,q)$ in $M$ 
is said to be \textit{laminar} if there exists a sequence of open subsets $\om_i\subset M$ and a sequence of currents $T_i$, respectively locally uniformly laminar in $\om_i$, such that $T_i$ increases to $T$ as $i\cv\infty$.
Equivalently, $T$ is laminar iff there exists a measured family $((\Delta_a)_{a\in \mathcal{A}}, m)$ 
of compatible holomorphic disks of dimension $q$ in $M$ such that $T = \int_\mathcal{A} [\Delta_a]dm(a)$. Here compatible means that the intersection of two disks in the family is relatively open (possibly empty) in each of the disks (i.e. the disks are analytic continuations of each other). It is important to note that for laminar currents there is no   control on the geometry of the disks (even locally). These geometric currents  appear rather  frequently in holomorphic dynamics. The reader is referred to 
\cite{bls} for a general account on this notion (see also \cite{maximal, cantat survey}). 

\medskip

Why should we wonder  about the laminarity of the bifurcation currents? We have been emphasizing the fact that $\supp(\tbif^k)$ is in a sense the  locus of ``bifurcations of order $k$''. If true, this would mean that on $\supp(\tbif^k)\setminus \supp(\tbif^{k+1})$ we should see ``stability in codimension $k$'', that is, we should expect   
$\supp(\tbif^k)\setminus \supp(\tbif^{k+1})$ to be filled 
with submanifolds of codimension $k$ where the dynamics is stable.  There is a natural stratification of parameter space according to the dimension of the space of deformations, and our purpose is to compare this stratification with that of the supports of the successive bifurcation currents. Laminarity is the precise way to formulate this problem. 

Let us be more specific.  Throughout this paragraph $\La$ is either the space of polynomials     or the space of rational maps of degree $d$, with marked critical points if needed. We let $D = \dim(\La)$.
We say that two rational maps are deformations  of each other if there is a $J$-stable   family
connecting them\footnote{Notice that this is weaker than the notion considered in \cite{mcms} (stability over the whole Riemann sphere), which  introduces some distinctions which are not relevant
from our point of view, like distinguishing the center from the other parameters in 
 a hyperbolic component.}. 

 In \cite{mcms} McMullen and Sullivan   ask for a general description of the way the deformation space of a given rational map embeds in parameter space.  Our thesis is that these submanifolds tend to be organized into laminar currents\footnote{We do not address the problem of the global holonomy of these laminations, which gives rise to interesting phenomena \cite{branner-turning}.}. 

\medskip

We start with a few general facts. The first easy proposition asserts  that if $\tbif^k$ is laminar on some set of positive measure outside $\supp(\tbif^{k+1})$, then the corresponding disks are indeed disks of deformations.

\begin{prop}
Assume that $S$ is a   laminar current of bidegree $(k,k)$ such that $S\leq \tbif^k$ and $\supp(S)\cap \supp(\tbif^{k+1})=\emptyset$. Then the disks subordinate to $S$ are disks of deformations.
\end{prop}

\begin{proof}
By definition, a holomorphic disk of codimension $k$ is said to be subordinate to $S$ if there exists a non zero locally uniformly laminar current $U\leq S$ such that $\Delta$ is contained in a leaf of $U$. Observe that since $T_{\rm bif}^{k+1}$ gives no mass to analytic sets, there are no isolated leaves in $\supp(U)$.
Since $U\leq \tbif^k$ and $\supp(U)$ is disjoint from  $\supp(\tbif^{k+1})$, we see that $U\wedge \tbif =0$. Now if $\chi$ was not pluriharmonic along $\Delta$, then by continuity, it wouldn't be harmonic on the nearby leaves of $U$, implying that $U\wedge \tbif$ would be non-zero, a contradiction.
\end{proof}

\begin{prop}
For almost every parameter relative to the trace measure of 
$\tbif^k$, the dimension of the deformation space is at most  $D-k$. 

In particular a $\mubif$ generic parameter is rigid, that is it admits no deformations.
\end{prop} 

We should expect the codimension to be a.e. equal to $k$ on $\supp(\tbif^k)\setminus \supp(\tbif^{k+1})$. 

\begin{proof}
Similarly to Theorem \ref{thm:buff epstein}, a slicing argument shows that is enough to consider the case of maximal codimension, that is, to show that for $\tbif^D =\mubif$-a.e. parameter the deformation space is zero dimensional. 

If $f_0\in \supp(\mubif)$ possesses  a disk $\Delta$ of deformations, then $\chi\rest{\Delta}$ is harmonic.  An already mentioned pluripotential theoretic lemma \cite[Lemma 6.10]{fs2} asserts that if $E$ is a measurable  set with the property that through every point of $E$ there exists a holomorphic disk along which $\chi$ is harmonic, then $(dd^c\chi)^{D} (E)= \mubif(E) =0$. The result follows. 
\end{proof}

Let now $\la_0\in \supp(\tbif^k)$. There are at least $k$ active critical points at $f_{\la_0}$. Assume the number is exactly $k$, say $c_1, \ldots, c_k$, so the remaining  $D-k$ are passive ones. 
One might expect  each of these passive critical points to give rise to a modulus of deformations of $f_{\la_0}$, but there is no general construction for this. 

If these $D-k$ passive critical points lie in attracting basins,   the existence of $D-k$ moduli of deformation for $f_{\la_0}$ should follow from classical quasi-conformal surgery techniques. 

Observe that if the hyperbolicity conjecture holds, then $c_{k+1}, \ldots, c_D$ must be attracted by cycles.  Indeed let $U$ be an open set where these points are passive. By Theorem \ref{thm:tbifk per}, there exists $\la\in U$ such that $c_1(\la), \ldots, c_k(\la)$ are periodic. Hence there is an open set $U'\subset U$ where all critical points are passive. Assuming the hyperbolicity conjecture, in $U'$ all critical points lie in attracting basins. Thus this property persists for $c_{k+1}, \ldots, c_D$ throughout $U$.  

When $D=2$ and $k=1$, one can indeed construct these deformations and relate them to the geometry of $\tbif$. The following is a combination of results of   Bassanelli-Berteloot \cite{bas-ber2} and  the author \cite{cubic}.

\begin{thm}\label{thm:unif lamin}
If $\La = \mathcal{P}_3$ or $\mathcal{M}_2$ and if $U$ is an open set where one critical point is attracted by a cycle, then $\tbif$ is  locally uniformly laminar in $U$. 
\end{thm}

In the general case one is led  to the following picture.

\begin{conjecture}\label{conj:lamin}  
If $U\subset \La$ is an open set where $D-k$ critical points (counted with multiplicity)  lie in attracting basins, then $\tbif^k$ is a laminar current in $U$,  locally uniformly laminar outside a closed analytic set.
\end{conjecture}

The necessity of   an analytic subset where the uniform laminar structure might have singularities is due to the possibility of (exceptional) critical orbit relations

\medskip

To address the question of laminarity of $\tbif^k$ outside $\supp(\tbif^{k+1})$, we also need to analyze the structure of $\tbif^k$ in the neighborhood of the parameters lying outside $\supp(\tbif^{k+1})$ but having  more than $k$ active critical points. As Example \ref{ex:douady} shows, these parameters can admit deformations.
 There does not seem to be any reasonable understanding of the bifurcation current near these parameters, even for cubic polynomials.  Another interesting situation is that of cubic polynomials with a Siegel disk $\Delta$, such that a critical point falls in $\Delta$ after iteration. Of course in $\mathcal{P}_3$ both critical points are active. These parameters can be deformed by moving the  critical point in the Siegel disk \cite{zakeri}. We do not know whether these parameters belong to $\supp(\mubif)$. 
 
 \medskip

We also need to consider the possibility of  ``queer" passive critical points (whose existence contradicts the hyperbolicity conjecture, as seen above). In this case we have a positive 
result \cite{cubic}.

\begin{thm}\label{thm:queer lamin}
If $\La = \mathcal{P}_3$ and $U$ is an open set where one critical point is passive, then 
$\tbif$ is  laminar in $U$.
\end{thm}

Perhaps unexpectedly, this result does not follow from the construction 
of some explicit deformation for a cubic polynomial with   one active and one (queer) passive critical point.  Instead we use a general     laminarity criterion
  due to De Thélin \cite{dt-boule}: if a closed positive  current  $T$ in $U\subset \cd$ is the limit in the sense of currents of a sequence of integration currents   $T = \lim  {d_n}\inv [C_n]$ with $\mathrm{genus}(C_n) =O(d_n)$, then $T$ is laminar. We refer to \cite{cubic} for the construction of the curves $C_n$.

Thus in the space of cubic polynomials, Theorems \ref{thm:unif lamin} and \ref{thm:queer lamin} show that $\tbif$ is laminar outside the locus where two critical points are active (which is slightly larger than 
$\supp(\mubif)$). In the escape locus $\cd\setminus \mathcal{C}$, where $\tbif$ is uniformly laminar by Theorem  \ref{thm:unif lamin}, we can actually give a rather precise description  of $\tbif$, which   nicely  complements  the topological description given by Branner and Hubbard \cite{branner-hubbard1, branner-hubbard2}. 

We are also able to show that this laminar structure   really degenerates when approaching  $\supp(\mubif)$, in the sense that there cannot exist a set of positive transverse measure of deformation disks  ``passing through" $\supp(\mubif)$\footnote{Milnor discusses in \cite[\S 3]{milnor cubic} the possibility of so-called ``product configurations" in the connectedness locus of real cubic polynomials. Our result actually  asserts that in the complex setting 
such configurations cannot be of positive $\mubif$ measure.}.

A consequence of this is that the genera of the curves $\percrit(n,k)$ must be asymptotically larger than $3^n$ near $\supp(\mubif)$. We refer to \cite{cubic} for details.  The geometry of these curves was studied by Bonifant, Kiwi and Milnor \cite{milnor smooth, bkm} in a series of papers (see also the figures in \cite[\S 2]{milnor cubic} for some visual evidence of the complexity 
of the $\percrit(n,k)$ curves).

\medskip

The results in this paragraph suggest the following alternate characterization of the support of the bifurcation measure. 

\begin{question} 
If $\La$ is the moduli space of polynomials or rational maps of degree $d\geq 2$, is $\supp(\mubif)$ equal to the closure of the  set of rigid parameters?
\end{question}

\section{Bifurcation currents for families of Möbius subgroups}\label{sec:kleinbif}

The famous Sullivan dictionary provides a deep and fruitful analogy between the  dynamics of rational maps on $\pu$ and that of Kleinian groups \cite{sullivan1}. In this section we  explain that  bifurcation currents make also sense on the Kleinian group side, leading to interesting new results.   Unless otherwise stated, all results are due to   Deroin and the author \cite{kleinbif, bers slice}.

\subsection{Holomorphic families of subgroups of $\PSL$}
Here we gather some preliminary material and present the basic bifurcation theory of Möbius subgroups. We refer to the monographs of Beardon \cite{beardon} and Kapovich \cite{kapovich} for basics on the theory of Kleinian groups.
Throughout this section, $G$ is a finitely generated group and $\La$ a connected complex manifold. We consider a holomorphic family of representations of $G$ into $\PSL$, that is  a mapping $\rho:\La
\times G\cv \PSL$, such that for fixed  $\la\in \La$, $\rho(\la, \cdot)$  is a group homomorphism, and for fixed $g\in G$ 
  $\rho(\cdot, g)$ is holomorphic. The family will generally be denoted by  $(\rho_\la)_{\la\in \La}$. 

We make three standing assumptions:
\begin{enumerate}
\item[(R1)] the family is non-trivial, in the sense that there exists $\la_1$, $\la_2$ such that 
the representations  $\rho_{\la_i}$, $i=1,2$ are not   conjugate  in $\PSL$;
\item[(R2)] there exists  $\la_0\in \La$ such that  $\rho_{\la_0}$ is faithful;
\item[(R3)] for every $\la\in \La$, $\rho_\la$ is non-elementary. 
\end{enumerate}

Assumptions (R1) and (R2) do not really restrict our scope: this is obvious for (R1), and for (R2) is suffices to take a quotient of $G$. Notice that under (R2), the representations $\rho_\la$ are \textit{generally faithful}, that is, the set of of non faithful representations is a union of Zariski closed sets. 

Recall that a representation is said to be {\em elementary} when it admits a finite orbit on $\mathbb{H}^3 \cup\pu(\cc)$ ($\mathbb{H}^3$ is the 3-dimensional hyperbolic space).  Then, either $\Gamma$ fixes a point in $\mathbb{H}^3$ and is conjugate to a subgroup of 
$\mathrm{PSU}(2)$\footnote{or $\mathrm{SO}(3,\re)$ if we view $\mathbb{H}^3$ in its ball model.}  and in particular it contains only  elliptic elements, or it has a finite orbit (with one or two elements) on $\pu$. It can easily be proved that the subset of elementary representations of a given family $(\rho_{\la})$  is a real analytic subvariety $E$ of $\La$. Hence (R3) will be satisfied upon restriction to $\La\setminus E$ \footnote{In order to study the space of all representations of $G$ to $\PSL$, 
it is nevertheless interesting to understand which results remain true when allowing a proper subset of elementary representations. This issue is considered in \cite{kleinbif}, but  here for simplicity  we only work   with non-elementary representations.}. 

We identify $\PSL$ with the group of transformations of the form $\gamma(z)=\frac{az
+b}{cz+d}$, with $\lrpar{\begin{smallmatrix} a& b \\ c&d\end{smallmatrix}}\in \SL$
  and let 
$$\norm{\gamma} = \sigma(A^*A)^{1/2} \text{ and }
\tr ^2{\gamma} = (\tr A)^2, $$
where $A= \lrpar{\begin{smallmatrix} a& b \\ c&d\end{smallmatrix}}$ is a  lift of $\gamma$ to $\SL$, and $\sigma(\cdot)$ is the spectral radius.  Of course these quantities do not depend on the choice of the lift.

As  it is well-known,    Möbius transformations are classified into three types according to the value of their trace: 
\begin{itemize}
\itm {\em parabolic} if $\tr^2(\gamma)=4$ and $\gamma \neq\mathrm{id}$; it is then conjugate to  $z\mapsto z+1$;
\itm {\em elliptic} if $\tr^2(\gamma)\in [0,4)$, it is then conjugate to $z\mapsto e^{i\theta}z$ for some real number $\theta$, and $\tr^2{\gamma} = 2+2\cos(\theta)$.
\itm {\em loxodromic} if $\tr^2(\gamma)\notin [0,4]$, it is then conjugate to $z\mapsto kz$, with $\abs{k}\neq 1$.
\end{itemize}

There is a well-established notion of bifurcation for a family of Möbius subgroups, which is the translation in the Sullivan dictionary of Theorem \ref{thm:mss}. It has its roots in the work of  Bers, Kra,  Marden, Maskit, Thurston, and others on deformations of Kleinian groups.

\begin{thm} [Sullivan \cite{sullivan} (see also Bers \cite{bers})]\label{thm:sullivan}
Let $(\rho_\lambda)_{{\lambda\in \La}}$ be a holomorphic family of representations of $G$ into $\PSL$ satisfying (R1-3), and let $\om\subset\Lambda$ be a connected open set. Then the following assertions are equivalent:

\begin{enumerate}[{ i.}]
\item  for every $\lambda\in \om$, $\rho_\la(G) $ is discrete;
\item for every $\lambda\in \om$, $\rho_\lambda$ is faithful;
\item for every $g$ in $G$, if for some $\lambda_0\in \om$,  $\rho_{\lambda_0}(g)$ is loxodromic (resp. parabolic, elliptic), then $\rho_{\lambda}(g)$ is loxodromic
(resp. parabolic, elliptic) throughout $\om$;
\item for any $\lambda_0$, $\lambda_1$ in $\om$
the representations $\rho_{\lambda_0}$ and $\rho_{\lambda_1}$ are  quasi-conformally conjugate on $\pu$.
\end{enumerate}
\end{thm}

If one of these conditions is satisfied, we say that the family is \textit{stable} in $\om$. We define $\stab$ to be the maximal such open set, and $\bif$ to be  its complement, so that $\La = \stab\cup\bif$. 

  Theorem  \ref{thm:sullivan} shows that $\stab = \Int(\mathrm{DF})$ is the interior of the set of discrete and faithful representations. 
One main difference with rational dynamics is that, as a consequence of the celebrated
 J\o rgensen-Kazhdan-Margulis-Zassenhauss 
theorem, $\mathrm{DF}$ is a closed subset of parameters space. This is also referred to as Chuckrow's theorem, see \cite[p. 170]{kapovich}. This implies that, whenever non-empty, $\bif$ has non-empty interior, which is in contrast with Theorem \ref{thm:dense}. 

The following corollary is immediate:

\begin{cor}\label{cor:accidental}
 For every $t\in [0,4]$, the set of such parameters $\lambda_0$  at which there exists  $g\in G$ such that   
$\tr^2{\rho_{\lambda_0}(g)}=t$ and  $\lambda \mapsto \tr^2{\rho_{\lambda}(g)}$ is not locally constant, 
is dense in $\bif$.
\end{cor}

Again, a basic motivation for the introduction of bifurcation currents is the study of the asymptotic distribution of such parameters. 
The most emblematic value of $t$ is $t=4$. In this case one   either gets ``accidental'' new relations or new parabolic elements. Notice that when $t = 4\cos^2(\theta)$ with $\theta\in \pi\mathbb{Q}$ (e.g. $t=0$), then if $\tr^2{\rho_{\lambda_0}(g)}=t$, $g$ is of finite order, so these parameters also correspond to accidental new relations in $\rho_\la(G)$. 

A   famous result in this area of research is a theorem by McMullen \cite{mcm cusps} which asserts that accidental parabolics are dense {\em in the boundary} of certain components of stability. One might also wonder what happens of Corollary \ref{cor:accidental} when additional assumptions are imposed on $g$. Here is a question (certainly well-known to the experts) which was communicated to us by McMullen: if $G = \pi_1(S, \ast)$ is the fundamental group of a surface of finite type, does Corollary \ref{cor:accidental} remain true when restricting to the elements $g\in G$ corresponding to simple closed curves on $S$?

\medskip

Another important feature  of the space  of all  representations of $G$ into $\PSL$ (resp. modulo conjugacy) is that it admits a natural action of the automorphism group of $\mathrm{Aut}(G)$ (resp. the outer automorphism group $\mathrm{Out}(G)$). 
Despite recent advances, the dynamics of this action is 
not well understood (see the expository papers of Goldman \cite{goldman} and 
 Lubotzky \cite{lubotzky} for an account on this topic).  There is a promising interplay between these issues and  holomorphic dynamics, which was recently illustrated by the work of  Cantat \cite{cantat bers}.

\subsection{Products of random matrices}
To define a bifurcation current we use a  notion of   Lyapunov exponent of a representation, arising  from  a random walk on $G$. The properties of this Lyapunov exponent   will be studied using the theory of random walks on groups and
random products of matrices (good references on these topics are \cite{bougerol-lacroix, furman}). 

\medskip

Let us  fix a probability measure $\mu$ on $G$, satisfying the following assumptions:
 \begin{enumerate}
\item[(A1)] $\supp(\mu)$ generates $G$ as a semi-group;
\item[(A2)] there exists  $s>0$ such that  $\int_G \exp(s\;  \length (g)) d\mu(g) <\infty$.
\end{enumerate}
The length in (A2) is relative to the choice of any finite system of generators of $G$; of course the validity of (A2) does not depend on this choice. A typical case where these assumptions are satisfied is that of a finitely supported measure on a symmetric set of generators (like in the case of the simple random walk on the associated Cayley graph).  

Loosely speaking, the choice of  such a measure a measure on $G$ is somehow similar to the choice of a time parameterization for a flow, or more generally of a Riemannian metric along the leaves of a foliation. 

\medskip

We denote by  $\mu^n$ the $n^{\rm th}$ convolution power of $\mu$, that is the image of  $\mu^
{\otimes n}$ under $(g_1, \ldots ,g_n)\mapsto g_1\cdots g_n$.  This is the law of the  $n^{\rm th}$  step of the (left- or right-) random walk on $G$ with transition probabilites  given by  $\mu$.

If $\mathbf{g} = (g_n)_{n\geq 1}\in G^\nn$,  we let  $l_n(\mathbf{g}) = g_n\cdots g_1$ be 
the product on the left of the  $g_k$ (resp. $r_n(\mathbf{g}) = g_1\cdots g_n$ the product on the right). We denote by  $\mu^\nn$ the product measure on  $G^\nn$  so that  $\mu^n = (l_n)_*\mu^\nn = (r_n)_*
\mu^\nn$. 

\medskip

The {\em Lyapunov exponent} of a representation $\rho
: G \rightarrow \PSL$ is defined by the formula
\begin{equation}\label{def:lyapunov exponent} \chi (\rho ) : = \lim _{n\rightarrow \infty} \frac{1}{n}\int _G 
\log \norm{\rho(g)} d\mu^{n}(g) = \lim _{n\rightarrow \infty} \frac{1}{n}\int \log \norm{\rho(g_1\cdots 
g_n)} d\mu(g_1)\cdots d\mu(g_n), \end{equation}
in which the limit exists by sub-addivity. It immediately follows from Kingman's sub-additive ergodic theorem that 
\begin{equation} \label{eq: lyapunov a.e.}
\text{for }\mu^\nn \text{-a.e. }\mathbf{g}, \  \lim _{n\rightarrow \infty} \frac{1}{n} \log \norm
{\rho (l_n({\bf g}))} =\chi (\rho)  \end{equation}
(this was before Kingman a theorem due to Furstenberg and Kesten \cite{furstenberg kesten}).

The following fundamental theorem is due to Furstenberg \cite{furstenberg}:

\begin{thm}[Furstenberg]\label{thm:furstenberg}
Let $G$ be a finitely generated group and $\mu$ a probability measure on $G$   satisfying  (A1-2). Let 
$\rho:G\cv\PSL$ be a non-elementary representation.
 Then the Lyapunov exponent  $\chi(\rho)$  is positive and depends continuously on $\rho$.
\end{thm}

The next result  we need is due to Guivarc'h \cite{guivarch}:

\begin{thm}[Guivarc'h]\label{thm:guivarch}
Let $G$ be a finitely generated group and $\mu$ a probability measure on $G$   satisfying  (A1-2). Let 
$\rho:G\cv\PSL$ be a non-elementary representation.
Then 
 for  $\mu^\nn$-a.e $\mathbf{g}$, we have that 
 \begin{equation}\label{eq:guivarch}
\unsur{n}\log \abs{\tr (\rho( l_n(\mathbf{g})))} = \unsur{n}\log\abs{ \tr(\rho(g_n\cdots g_1))} \underset{n
\cv\infty}\longrightarrow
\chi(\rho).
\end{equation}
\end{thm}

A trivial remark which turns out to be a   source  of technical difficulties, is that, as opposite to 
 \eqref{eq: lyapunov a.e.}, one cannot in general integrate with respect to $\mu^\nn$ in \eqref{eq:guivarch}.  The reason of course is that some words can have zero or very small trace. 
Conversely, if $h$ is a function on $\PSL$, which is bounded below and equivalent to $\log\abs{\tr(\cdot)}$ as the trace tends to infinity, then one can integrate with respect to $\mu^\nn$.

An example of such a function is given by  the spectral radius, and   under the assumptions of Theorem  \ref{thm:guivarch} we obtain that
\begin{equation}\label{eq:lambdamax}
 \chi(\rho) = \lim_{n\cv\infty}\unsur{n}\int  \log\abs{\sigma(\rho( g))} d\mu^n(g). 
\end{equation}

% dire un peu plus loin qu'on aura besoin d'une version précisée du théorème....

\subsection{The bifurcation current}
We are now ready for the introduction of the bifurcation current, following \cite{kleinbif}. Let $G$ be a finitely generated group, and $(\rho_\la)_{\la\in \La}$ be a holomorphic family of representations satisfying (R1-3). Fix a probability measure $\mu$ on $G$ satisfying (A1-2). It follows immediately from \eqref{def:lyapunov exponent} that $\chi:\la\mapsto\chi(\rho_\la)$ 
is a psh function on  $\La$. Motivated by the analogy with rational dynamics, it is natural to suggest the following definition. 

\begin{defi}\label{def:bifcur}
Let $(G,\rho,\mu)$ be as above.  The bifurcation current associated to $(G,\mu,\rho)$ is
$\tbif  = dd^c\chi$. 
\end{defi}

At this point it is still unclear whether this gives rise to a meaningful concept. This will be justified by the following results. 

First, it is easy  to see that $\supp(\tbif)$ is contained in the bifurcation locus. Indeed, if $(\rho_\la)$ is stable on $\om$, then the Möbius transformations $\rho_\la(g)$, $g\in G$, do not change type throughout $\om$. In particular for every $g\in G$, $\om\ni\la\mapsto \log\abs{\sigma(\rho_\la(g))}$ is pluriharmonic. We thus infer from \eqref{eq:lambdamax} that $\chi$ is pluriharmonic on $\om$. 

It is a remarkable fact that conversely, pluriharmonicity of $\chi$ characterizes stability:

\begin{thm}\label{thm:support}
Let $(G, \rho, \mu)$ be a holomorphic family of representations of $G$, satisfying (R1-3), endowed with a measure $\mu$ satisfying (A1-2). Then $\supp(\tbif)$ is equal to the bifurcation locus. 
\end{thm}

Here is a sketch of the proof, which consists in several steps, and involves already encountered arguments. It is no loss of generality to assume that $\dim(\La) =1$. We need to show that if $\om$ is an open set disjoint from $\supp(\tbif)$, then $(\rho_\la)$ is stable in $\om$. 
 The main idea is to look for a geometric interpretation of $\tbif$, in the spirit of what we did in \S\ref{subs:marked}. 
 
 For this we need a substitute for the equilibrium  measure of a rational map: this will be the unique stationary measure under the action of 
$\rho_\la(G)$ on $\pu$. Let us be more specific. For every representation, 
$(G,\mu)$ acts by convolution on the set of probability measures on $\pu$ by the assignment
$$\nu\longmapsto \int \rho(g)_* \nu \; d\mu(g).$$
   Any fixed point of this action is called a {\em stationary measure}. The following theorem is intimately related to Theorem \ref{thm:furstenberg}. It is in a sense the analogue of the Brolin-Lyubich theorem in this context.
   
 \begin{thm}
 Let $G$ be a finitely generated group and $\mu$ a probability measure on $G$   satisfying  (A1-2). Let 
$\rho:G\cv\PSL$ be a non-elementary representation. Then   
     there exists a unique stationary probability measure $\nu$ on $\pu$. Furthermore, for any $z_0\in \pu$, the sequence  $\int  ({\rho_{\la_0}(g)})_*\delta_{z_0} d\mu^n(g)$ converges to $\nu$. 
\end{thm}

In analogy with \S\ref{subs:marked},  let us now    work   in 
$\La\times \pu$, and  consider the fibered action $\widehat g$ on $\La\times \pu$ defined by $\widehat g:(\la,z)\mapsto (\la, \rho_\la(g)(z))$. We  seek for a current $\widehat T$ on $\La\times \pu$ ``interpolating'' the stationary measures.  Given $z_0\in \pu$ , 
we then  introduce the  sequence of positive closed currents $ \widehat{T}_n$ defined by 
\begin{equation}\label{eq:Tn hat}
 \widehat{T}_n = \unsur{ n} \int \left[\widehat{g}\left(\La\times\set{z_0}\right) \right] d\mu^n(g).
\end{equation}

To understand why it is natural to have a linear    normalization in \eqref{eq:Tn hat}, think of a polynomial family of representations. Precisely, assume that $\La=\cc$ and that 
for a set of  generators $g^1, \ldots, g^k$ of $G$, the matrices $\rho_\la(g^j)$ are polynomial in $\la$. Then if $w$ is a word in $n$ letters in the generators, $\rho_\la(w)$ has degree $O(n)$ in $\la$, hence in 
$\cc\times \pu\subset\pu\times \pu$, the 
degree of the graph $\widehat{w}(\La\times \set{z_0})$   is $O(n)$. 

It turns out that this sequence does {\em not} converge to a current interpolating the stationary measures (since e.g. it vanishes above the stability locus) nevertheless it   gives another crucial information.

\begin{prop}
 The sequence of currents  $\widehat{T}_n$  converges to $\pi_1^*\tbif$, where $
\pi_1:\La\times\pu\cv\La$ is the natural projection. 
\end{prop}
 
Figure \ref{fig:hat} is  a visual interpretation of this result. 

\begin{figure}[h]\label{fig:hat}
\begin{center}
  \psfrag{L}{$\La$}
  \psfrag{P}{$\pu$}
  \psfrag{S}{$\stab$}
  \psfrag{B}{$\bif$}
  \psfrag{z}{\small $\La\times\set{z_0}$}
  \psfrag{g}{\small $\widehat{g_1\cdots g_n}(\La\times\set{z_0})$}
\includegraphics[scale=0.4]{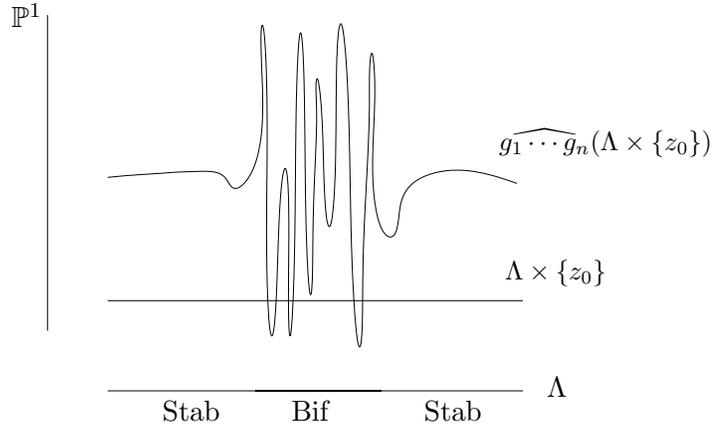}
\caption{Fibered action of $g_1\cdots g_n$} 
\end{center}
\end{figure}

This   proposition implies  that if $\om$ is disjoint from $\supp(\tbif)$, the average growth of the area of the sequence of graphs $\left(\widehat{g_1\cdots g_n}\right)\left(\La\times\set{z_0}\right)$ over $\om$ is sublinear in $n$.
The attentive reader will have noticed   that the situation   is similar to that of Section \ref{sec:prologue}, except that the sequence of graphs over $\La$ is replaced by an average of graphs. The next result is in the spirit of  Proposition \ref{prop:support}.

\begin{prop}
 If $\om$ is an open subset of $\La$ such that $\overline\om$ is compact and 
 disjoint from  $\supp(\tbif)$, then the mass of $\widehat{T}_n$ in $\om$ is $O\lrpar{\unsur{n}}$. In other words: 
$$  \int \mathrm{Area}\lrpar{  \widehat{g}\left(\La\times\set{z_0}\right) \cap \pi_1^{-1}(\om) }d\mu^n(g) = O
(1).$$
\end{prop}

To prove this estimate we give an analytic expression of this area:
\begin{align*}
\int \mathrm{Area}\big(  \widehat{g}\left(\La\times\set{z_0}\right) & \cap \pi_1^{-1}(\om) \big) d\mu^n(g) =
\int_{\pi_1^{-1}(\om)}\big[\widehat{g}\left(\La\times\set{z_0}\right)\big]\wedge (\pi_1^*\omega_\La + 
\pi_2^*\omega_\pu)d\mu^n(g) \\
&=\mathrm{Area}(\om) + \int_\om (\pi_1)_* \left(\pi_2^*\omega_{\pu}\rest{\widehat{g}\left(\La\times\set
{z_0}\right)}\right) d\mu^n(g) \\
&= \mathrm{Area}(\om)+n \int_\om dd^c\chi_n   \text{ , where } \chi_n =  \unsur{n} \int\log\frac{\norm{\rho_\la(g)
(Z_0)}}{\norm{Z_0}} d\mu^n(g).
\end{align*}
Here $Z_0$ is a lift of $z_0$ to $\cd$, $\norm{\cdot}$ is the Hermitian norm, and the $dd^c$ is taken w.r.t. $\la$. Therefore, exactly as in Theorem \ref{thm:activity support} we are left to prove that $\chi_n-\chi = O\lrpar{\unsur{n}}$. This estimate in turn follows from ergodic theoretic properties of random matrix products --specifically, from the ``exponential convergence of the transition operator'', a result due to Le Page \cite{lepage}. 

\medskip

The next step is to combine this estimate on the average area of the graphs over $\om$ with  another  result of Furstenberg which asserts that for each fixed $\la$,   for  $\mu^\nn$-a.e. sequence  $\mathbf{g}=(g_n)$, $\rho_\la(g_1\cdots g_n)({z_0})$ converges  to a point $z_{\mathbf{g}}$ (a consequence of the martingale convergence theorem). From this, as in Lemma \ref{lem:dim1},
  we infer that for  a.e.   $\mathbf{g} = (g_n)$, the sequence of graphs  $\left(\widehat
{g_1\cdots g_n}\right)\left(\La\times\set{z_0}\right)$ converges   to a limiting graph $\Gamma_{\mathbf
{g}}$, outside a finite set of vertical bubbles. We then obtain a  measurable and $G$-equivariant family  of limiting graphs over $\om$. Formally, this family is parameterized by the \textit{Poisson boundary $P(G, \mu)$} of $(G,\mu)$.

To show that the family of representations is stable over $\om$, we  need to upgrade this family of graphs into a holomorphic motion of the Poisson boundary, that is, we need to show that two distinct graphs are disjoint. For this, we use the fact that the number of intersection points between two graphs over some domain $D\subset \om$ is a function on  
$P(G, \mu)\times P(G,\mu)$ satisfying certain invariance properties. 
By an ergodic theorem due to Kaimanovich   \cite{kaimanovich}, this function is a.e. equal\footnote{As stated here, the result is true only when $\mu$ is invariant under $g\mapsto g\inv$. 
The general case needs some adaptations. Notice also that the Kaimanovich theorem can be viewed as a far reaching generalization of the ergodicity of the geodesic flow for manifolds of constant negative curvature. } to a
 constant $\iota_D$ on $P(G, \mu)\times P(G,\mu)$, depending only on $D$ . The assignement $D\mapsto \iota_D$ being integer valued, we infer that there exists a locally finite set of points $F$  such that if $D\cap F = \emptyset$, then $\iota_D=0$. On such a $D$,  
it follows 
that the family of graphs is a holomorphic motion, and ultimately, 
  that the family is stable. Finally, to show that the exceptional set $F$ is empty and    conclude that the family is stable on all $D$, we use the  fact that the set of discrete faithful representations is closed, 
 which implies that $\bif$ cannot have isolated points. \qed

\subsection{Equidistribution of representations with an element of    a given trace}\label{subs:trace}
Another important aspect of the bifurcation currents is that they enable to obtain equidistribution results in Corollary \ref{cor:accidental}. This is the analogue in our context  of the results of \S \ref{subs:multiplier}.
 For $t\in \cc$ denote by  $Z(g,t)$ the analytic subset of 
 $\La$ defined by  $Z(g,t) = \set{\la, \ \tr^2(\rho_\la(g)) = t}$.  We study the asymptotic properties of the integration currents $[Z(g,t)]$. Note that if $\tr^2(\rho_\la(g))\equiv t$, then
 $Z(g,t)=\La$, and by convention, $[\La]=0$.  

The first equidistribution result is the following.

\begin{thm}[Equidistribution for random sequences]\label{thm:equidist1}
Let $(G,\rho,\mu)$   be a holomorphic family of representations of $G$, satisfying (R1-3), endowed with a measure $\mu$ satisfying (A1-2). Fix
  $t\in\cc$. Then for  $\mu^\nn$ a.e. sequence 
$(g_n)$,  we have that 
$$\unsur{2n}\left[Z(g_n\cdots g_1, t )\right] \underset{n\cv\infty}\longrightarrow
 T_{\rm bif}.$$
 \end{thm}
  
  Notice that, if instead of considering a random sequence in the group, we take a word obtained by applying to $g\in G$ an iterated element of $\mathrm{Aut}(G)$, then similar equidistribution results were obtained by Cantat in \cite{cantat bers} 
  
  \medskip
  
The following ``deterministic'' corollary makes
 Corollary \ref{cor:accidental} more precise. It seems difficult to prove it without using probabilistic methods. 
 
 \begin{cor}
Under the assumptions of Theorem \ref{thm:equidist1}, let  $\e>0$ and  $\La'\Subset \La$. Then there exists $g\in G$ such that 
$\la\mapsto \tr^2(\rho_\la(g))$ is non-constant and  $\set{\la, \tr^2(\rho_\la(g))=t}$ is $\e$-dense in 
$\bif\cap \La'$.
 \end{cor}
 
 For the value $t=4$, it is unclear which of 
 accidental relations or parabolics prevail in   $\left[Z(g_n\cdots g_1, 4)\right ]$.
  
  \medskip
  
 Theorem \ref{thm:equidist1} is actually a consequence of a more general theorem \cite[Thm 4.1]{kleinbif}, which gives rise to several other random equidistribution results.
 
  \medskip
  
As opposed to the case of rational maps, we are able to estimate the speed of convergence, up to some averaging on $g$ and a global assumption on $\La$. 

\begin{thm}[Speed in the equidistribution theorem]\label{thm:equidist speed}
Let $(G,\rho,\mu)$   be a holomorphic family of representations of $G$, satisfying (R1-3), endowed with a measure $\mu$ satisfying (A1-2). Fix
  $t\in\cc$. Assume further that one of the following hypotheses is satisfied:
   \begin{enumerate}
 \item[i.]   $\La$  is an algebraic family of representations defined over  $\overline{\mathbb{Q}}$.
 \item[ii.] There exists a geometrically finite representation in  $\La$.
 \end{enumerate}
Then there exists a constant $C$ such that for every test form  $\phi$,
$$\bra{\unsur{2n} \int  \left[Z(g,t)\right] d\mu^n(g) -T_{\rm bif}, \phi}\leq
 C \frac{\log n }{n} \norm{\phi}_{C^2}.
$$
 \end{thm}

The meaning of the notion of an algebraic family of representations is the following: the space $\mathrm{Hom}(G,\PSL)$ of representations of $G$ in $\PSL$ admits a natural structure of an affine algebraic variety over $\mathbb{Q}$, simply by describing it as a set of matrices satisfying certain polynomial relations\footnote{To view $\PSL$ as a set of matrices, observe that $\PSL$ is isomorphic to $\mathrm{SO}(3,\cc)$ by the adjoint representation.}. Changing this set of generators amounts to performing algebraic changes of coordinates, so that this structure of algebraic variety  is well-defined.  We say that an arbitrary family of representations, viewed as a holomorphic mapping $\rho:\La\cv\mathrm{Hom}(G,\PSL)$  is algebraic (resp. algebraic over $K$) if    $\rho\rest{\La}$ is a dominating map to some    algebraic subvariety (resp. over $K$) of $\mathrm{Hom}(G,\PSL)$. To say it differently,  there exists an open subset  $\om\subset \La$ such that $\rho_{\om}$ is an open subset of an algebraic subvariety of $\mathrm{Hom}(G,\PSL)$. .

\medskip

These results parallel those of \S \ref{subs:multiplier}, and exactly as in Theorem \ref{thm:basber}, they become much easier after some averaging with respect to the multiplier. Let us illustrate this by proving the following result:

\begin{prop}
Let $m$ be the normalized Lebesgue measure on  $[0,4]$. Then under the assumptions of Theorems  
\ref{thm:equidist1}, we have that 
$$\unsur{2n} \int  \left[Z(g,t)\right] d\mu^n(g)dm(t) \underset{n\cv\infty}\longrightarrow
 T_{\rm bif}.$$
\end{prop}

\begin{proof}  We prove the $L^1_{\rm loc}$ convergence of the potentials. Let $u(g,\cdot)$ be 
the psh function on $\La$ defined by 
$$u(g,\la) = \int \log \abs{\tr^2(\rho_\la(g))-t} dm(t) = v(\tr^2(\rho_\la(g))),$$ 
where $v$ is the logarithmic potential of  $m$ in $\cc$.  The function   $u$ is   bounded below and  
$u(g, \la ) \sim \log\abs{\tr^2 (\rho_\la(g))}$  as $\tr^2 (\rho_\la(g))$ tends to infinity.

If $\la$ is fixed, then by Theorem \ref{thm:guivarch}, for  $\mu^\nn$ a.e. sequence $(g_n)$, 
$\unsur{2n} u(g_1\cdots g_n, \la)\cv \chi(\la)$.    Since $u$ is bounded below, we can apply the dominated convergence theorem and integrate with respect to  $g_1,\ldots ,g_n$. We conclude that 
$$\unsur{2n} \int u(g,\la) d\mu^n(g)\underset{n\cv\infty}\longrightarrow \chi(\la)$$  for all  $\la\in \La$, which, by taking the $dd^c$ in $\la$, implies the desired statement. 
\end{proof}

\begin{proof}[Sketch of proof of Theorem \ref{thm:equidist1}]
We need to show that for a.e. sequence $(g_n)$,  the   sequence of psh functions $\unsur{2n}\log\abs{\tr^2(\rho_\la (g_1\cdots g_n)) -t}$ converges to $\chi$. The point is to find a choice of random sequence $(g_n)$  which does not depend on $\la$. For this,  a kind of  sub-additive ergodic theorem with values in the space of psh functions   shows that for a.e. $(g_n)$, $\unsur{n} \log \norm{\rho_\la (g_1\cdots g_n)}$ converges in $L^1_{\rm loc}$ to $\chi$. 
By  Theorem  \ref{thm:guivarch}, it is possible to choose the 
  sequence $(g_n)$   so that for $\la$ belonging to a countable dense sequence of parameters, 
$\unsur{2n}\log\abs{\tr^2(\rho_\la (g_1\cdots g_n)) -t}$ converges to $\chi(\la)$.  
  On the other hand, 
$$\unsur{2n}\log\abs{\tr^2(\rho_\la (g_1\cdots g_n)) -t}\leq \unsur{n} \log \norm{\rho_\la (g_1\cdots g_n)} + o(1).$$  Using the continuity of $\chi$ and 
  the Hartogs lemma, we conclude that 
$\unsur{2n}\log\abs{\tr^2(\rho_\la (g_1\cdots g_n)) -t}$  converges to $\chi$ in $L^1_{\rm loc}$. 
\end{proof}

For Theorem \ref{thm:equidist speed}, the main difficulty is that for a given parameter
 we cannot in general integrate with respect to 
$g_1, \cdots, g_n$ in the almost sure convergence 
$$
\unsur{2n}\log \abs{\tr^2(\rho_\la(g_1\cdots g_n)) - t} \cv \chi(\la),
$$ 
due to the possibility of elements with trace very close to $t$. This is exactly similar to the
 difficulty encountered in Theorem \ref{thm:basber}.{\em iii}. We estimate the size of the set of parameters where this exceptional phenomenon happens by using volume estimates for sub-level sets of psh functions and the global assumption {\em i.} or {\em ii.} 
 In both cases  this  global assumption is used to show the existence of 
  a parameter at which $ \abs{\tr^2(\rho_\la(g_1\cdots g_n)) - t}$ is not 
  super-exponentially small in $n$. Under {\em i.}, this follows from a nice number-theoretic lemma 
(a generalization of the so-called {\em Liouville inequality}), which was communicated to us by P. Philippon.
  Another key ingredient is a large deviations estimate in Theorem \ref{thm:guivarch}, which was obtained independently by Aoun \cite{aoun}.

\subsection{Canonical bifurcation currents}
One might object that our  definition of   bifurcation currents in spaces of representations lacks of naturality, for   it depends on the choice of a measure $\mu$ on $G$ --recall however from Theorem \ref{thm:support} that the support of the bifurcation current is independent of $\mu$. In this paragraph, following \cite{bers slice}, we briefly explain how 
 a canonical bifurcation current can be constructed under natural assumptions. 
 
Let $X$ be a compact Riemann surface of genus $g\geq 2$, and 
  $G = \pi_1(X, \ast)$ be its fundamental group. Let $(\rho_\la)$ be a  
  holomorphic family of representations of $G$ into $\PSL$ satisfying (R1-3). We claim that 
  there is a   Lyapunov exponent function on $\La$ which  is canonically associated to the Riemann surface structure of $X$ (up to a multiplicative constant). 
  
  For this, let $\widetilde{X}$ be the universal cover of $X$ (i.e. the unit disk). $G$ embeds naturally as a subgroup of $\mathrm{Aut}(\widetilde X)$. For any representation $\rho\in \La$, consider its suspension $X_\rho$, that is the quotient of $\widetilde{X}\times \pu$ by the    diagonal action of $G$. The  suspension is a fiber bundle over $X$, with $\pu$ fibers, and admits  a holomorphic foliation transverse to the fibers whose holonomy is  given by $\rho$. 
If  $\gamma$ is any path on $X$, we denote by  $h_\gamma$ its holonomy  $\pu_{\gamma(0)}\cv \pu_{\gamma(1)}$. 
 
The Poincaré metric endows $X$ with a  natural Riemannian structure, so we can consider the Brownian motion on $X$. It follows from  the sub-additive ergodic theorem that  for a.e. Brownian path $\om$ (relative to the Wiener measure), the limit  
$$\chi  (\omega) = \lim_{t\cv\infty} \unsur{t}\log\norm{h_{\omega(0), \omega(t)}}$$ 
  (where $\norm{\cdot}$ is any smoothly varying  spherical metric on the fibers) exists, and does not depend on $\om$. 

We define   $\chi_{\rm Brownian}(\rho)$ to be this number, and introduce a natural  bifurcation current on $\La$ 
by putting $\tbif = dd^c\chi_{\rm Brownian}$. We have the following theorem.

\begin{thm}
Let as above $(\rho_\la)$   be a holomorphic family of non-elementary representations of the fundamental group of a compact Riemann surface, satisfying (R1-3).  

Then the function  $\chi_{\rm Brownian}$ is psh on $\La$ and the support of  $\tbif= dd^c\chi_{\rm Brownian} $ is the bifurcation locus.
\end{thm}

To prove this theorem, it is enough to exhibit a measure $\mu$ on $G$ satisfying (A1-2) and such that for every $\rho$, 
$\chi_\mu(\rho) = \chi_{\rm Brownian}(\rho)$ (up to a multiplicative constant). Such a measure actually exists and was constructed using a    discretization procedure by Furstenberg \cite{furstenberg discretisation}. It is   non-trivial  to check that $\mu$ satisfies the exponential moment  condition (A2) (for instance this measure can never be of finite support). 

\medskip

There is another natural family of paths on $X$: the geodesic trajectories. An argument similar to the previous one shows that if $(x,v)\in S^1(X)$ (unit tangent bundle) is generic relative to the Liouville measure, and if $\gamma_{(x,v)}$ denotes the unit speed geodesic stemming from $(x,v)$, then the limit  $\lim_{t\cv\infty} \unsur{t}\log\norm{h_{\gamma
(0), \gamma(t)}}$ exists and does not depend on (generic)  $(x,v)$. We denote by  $\chi_{\mathrm{geodesic}}(\rho)$  this number. It follows from the elementary properties of the Brownian motion on the hyperbolic disk that there exists a constant $v$ depending only on $X$ such that $\chi_{\rm Brownian}  = v  \chi_{\mathrm{geodesic}}$. Therefore the associated bifurcation current is the same. 

\medskip

Here is a situation where these ideas naturally apply:  consider the set $\mathcal{P}(X)$ 
of complex projective structures over a Riemann surface $X$, compatible with its 
complex structure (see~\cite{dumas} for a nice introductory text on projective structures). 
This is a complex affine space of dimension $3g-3$,  
 admitting a distinguished point, the ``standard Fuchsian structure'', namely the projective structure obtained by viewing $X$ as a quotient of the unit disk. A projective structure induces  a {\em holonomy representation} (which is always non-elementary and defined only up to conjugacy) so the above discussion applies. We conclude that {\em the space of projective structures on $X$ admits a natural bifurcation current}.

From the standard Fuchsian structure, one classically constructs an embedding of the Teichm\"uller space of 
$X$ as a bounded open subset of $\mathcal{P}(X)$, known as the {\em Bers embedding} (or {\em Bers slice}). This open set can be defined for instance as the component of the distinguished point in the stability locus. 

In \cite{mcm book}, McMullen   suggests the Bers slice as the analogue of the Mandelbrot set through the Sullivan dictionary. From this perspective,   an interesting result in \cite{bers slice} is that
the canonical Lyapunov exponent function $\chi_{\rm Brownian}$ is {\em constant} on the Bers embedding, so the analogy also holds at the level of Lyapunov exponents.  
 
\subsection{Open problems}
There are many interesting open questions in this area, some of them stated in \cite[\S 5.2]{kleinbif}.  
In the spirit of this survey, let us only state two problems related to 
 the   exterior powers of $\tbif$. 

As opposite to the case of rational maps, we believe that the supports of $\tbif^k$ for $k\geq 2$ do not give rise to ``higher bifurcation loci''.

\begin{conjecture}
Let  $(G, \rho, \mu)$ be a family of representations 
satisfying (R1-3) and (A1-2), and  assume further that two representations in $\La$ are never conjugate in $\PSL$ (that is, $\La$ is a subset of the character variety). Then for every $k\leq \dim(\La)$, $\supp(\tbif^k) = \bif$.
\end{conjecture}

Here is some evidence for this conjecture: let $\theta \in \re\setminus\pi\mathbb{Q}$, $t=4\cos^2(\theta)$ and consider  the varieties $Z(g,t)$. Since for $\la\in Z(g,t)$, $\rho_\la$ is not discrete, the bifurcation locus of $\set{\rho_\la, \la \in Z(g,t)}$ is equal to $Z(g,t)$. Hence $\supp(\tbif\wedge [Z(g,t)]) = Z(g,t)$, which by equidistribution of $Z(g,t)$ makes the equality 
$\supp(\tbif^2) = \supp(\tbif)$ reasonable.

\medskip

It is also natural to  look for    equidistribution in higher codimension. Here is a specific question: 

\begin{question}
 Let $(G, \rho, \mu)$ be a family of representations  satisfying (R1-3) and (A1-2). Assume that $\dim(\La)\geq 3$. Given a generic element $h\in \PSL$, and a $\mu^\nn$ generic sequence $(g_n)$, do the solutions of the equation $g_1\cdots g_n = h$ equidistribute (after convenient normalization) towards $\tbif^3$?
\end{question}

% We show in \cite{kleinbif} that the support theorem \ref{thm:suppport} still holds when $\La$ admits a proper subset of elementary representations.  
% 
% \begin{question}
%  Do the equidistribution results of \ref{subs:equidist} remain true when $\La$ admits elementary representations (not forming the whole family?).
% \end{question}
% 
% Let us  illustrate one difficulty in this problem:  consider a situation where $\La$ is a one-dimensional family of representations, such that  $\rho_{\la_0}$ is an isolated elementary representation, with $\rho_{\la_0}(g)$ parabolic for every $g$. Since the Lyapunov exponent function is locally bounded, $\tbif$ is well defined on $\La$ and gives no mass to $\set{\la_0}$. On the other hand, if we want to prove that  $\unsur{2n}[Z(g,4)]$ converges to $\tbif$, 
%  we need to control the multiplicity of the solution $\la_0$ to the equation $\tr^2(g_\la)= 4$, which seems delicate.

\section{Further settings, final remarks}\label{sec:further}

In this section we gather some speculations about possible extensions of the results presented in the paper.

\subsection{Holomorphic dynamics in higher dimension}\label{subs:higher}
It is likely that a substantial part of the theory of bifurcation currents for rational maps on $\pu$  should remain true in higher dimension, nevertheless little has been done so far.  

\medskip

Let us first discuss the case of polynomial automorphisms of $\cd$. A polynomial automorphism $f$ of degree $d$ of $\cd$ with non-trivial dynamics admits a unique measure of maximal entropy, which has two (complex) Lyapunov exponents of opposite sign $\chi^+(f)>0>\chi^-(f)$ and describes the asymptotic distribution of saddle periodic orbits \cite{bls, bls2}. See \cite{cantat survey} in this volume for a presentation of these results for automorphisms of compact complex surfaces.
Notice that a polynomial automorphism has constant jacobian, so $\chi^+ (f)+ \chi^-(f) = \log\abs{\mathrm{Jac}(f)}$ is a pluriharmonic function on parameter space. It is not difficult  to see that the function 
$f\mapsto \chi^+(f)$ is psh (in particular upper semi-continuous), and  it was shown in \cite{lyapunov} that is actually continuous (even for families degenerating to a one-dimensional map). 

Since the Lyapunov exponents are well approximated by the multipliers of saddle orbits \cite{bls2},  it follows that
 near any point in  parameter space where $f\mapsto \chi^+(f)$ is not pluriharmonic, complicated bifurcations of saddle points occur. In the dissipative case they must become attracting. The main idea of Theorem \ref{thm:basber} seems robust enough to enable some generalization to this setting. 

On the other hand, a basic understanding of the 
 phenomena responsible for the bifurcations of a family of polynomial automorphisms 
of $\cd$ --e.g. the role of homoclinic tangencies-- is still missing (see   \cite{bsR} for some results in a particular case). In particular   no reasonable analogue of 
 Theorem \ref{thm:mss} is available for the moment. Therefore it seems a bit  premature to hope for a characterization of the support  of  the  bifurcation current  $dd^c\chi^+$, let alone $(dd^c\chi^+)^p$.

\medskip

The situation is analogous in the case of families of holomorphic endomorphisms of $\pk$ 
(and more generally for families of polynomial-like mappings in higher dimension). The 
regularity properties of  the Lyapunov exponent(s) function(s) are rather well understood, due to the work of Dinh-Sibony \cite{ds-pl} and Pham \cite{pham} (a good account on this is in \cite[\S 2.5]{ds-survey}). In particular it is known that the sum $L_p(f)$ of the $p$ largest Lyapunov exponents of the maximal entropy measure is psh for $1\leq p\leq k$ and the sum of all Lyapunov exponents is Hölder  continuous.  It is also known \cite{bdm} that $L_p(f)$ is well approximated by the corresponding quantity evaluated at repelling periodic cycles, so that 	  any point in parameter space where $L_p$ is not pluriharmonic is accumulated by bifurcations of periodic points.  Notice that the relationship between the currents $dd^cL_p$ is unclear. 

Again, one may reasonably hope for equidistribution results in the spirit of Theorem \ref{thm:basber}. 

Another interesting point is a formula, given in \cite{bas-ber1}, for the sum $L_k$ of Lyapunov exponents of endomorphisms of $\pk$ which generalizes Przytycki's formula \eqref{eq:lyap}. From this formula one may expect to reach some understanding on the role of the critical locus towards bifurcations. 
 
\subsection{Cocycles}
Yoccoz suggests in \cite{yoccoz} to study the geography of the (finite dimensional) space of locally constant $\mathrm{SL}(2,\re)$ 
cocycles over a transitive   subshift of finite type in the same way as spaces of one-dimensional holomorphic dynamical systems, with some emphasis on the description  of hyperbolic components and their  boundaries. 
 For $\SL$ cocycles 
 (and more generally for any cocycle with values in a complex  Lie subgroup of $\mathrm{GL}(n,\cc)$) we have an explicit connection with holomorphic dynamics given by the bifurcation currents. 
 Indeed, locally constant cocycles over a subshift are generalizations of random products of matrices, which correspond to cocycles over the full shift.  In this situation we can define a Lyapunov exponent function relative to a fixed measure on the base dynamical system (the Parry measure is a natural candidate), and construct a bifurcation current  by taking the $dd^c$. 
  
Notice that the subharmonicity properties of Lyapunov exponents are frequently used in this area of research (an early example is \cite{herman}).

For a general holomorphic family of (say, locally constant) $\SL$ cocycles over a fixed subshift of finite type,  one may ask the same questions as in Section \ref{sec:kleinbif}: characterize the support of the bifurcation current, prove equidistribution theorems. Another interest of considering  this setting  is that it is somehow  a simplified model of the tangent dynamics of 2-dimensional diffeomorphisms, so it might provide  some insight on the bifurcation theory of those. In particular there is an analogue of heteroclinic  tangencies in this setting (``heteroclinic connexions"), and it might be interesting to study the distribution of the corresponding parameters. 

\subsection{Random walks on other groups}
Another obvious possible  generalization of Section \ref{sec:kleinbif} is the  study of bifurcation currents associated to holomorphic families of finitely generated subgroups of $\mathrm{SL}(n,\cc)$. Again, if $(\rho_\la)$ is a holomorphic family of strongly irreducible representations (see \cite{furman} for the definition)  of a finitely generated group $G$ endowed with a probability 
 measure satisfying (A1-2), then Definition \ref{def:bifcur} makes sense, with $\chi$ being  
  the top Lyapunov exponent. It is likely that equidistribution theorems for representations possessing an element of given trace should follow as in 
\S \ref{subs:trace}. More generally, one may investigate the distribution of representations with an element belonging to a given hypersurface of $\mathrm{SL}(n,\cc)$, in the spirit of Section \ref{sec:prologue}. 

On the other hand, for the same reasons as in \S\ref{subs:higher}, the characterization of the support of the bifurcation current is certainly a more challenging problem. 

\subsection{Non-archimedian dynamics} It is a standard fact in algebraic geometry that studying families $(X_\la)_{\la\in \La}$ of complex algebraic varieties often amounts to studying varieties over a field extension of $\cc$, that is a function field in the variable $\la$. The same idea applies in   the dynamical context and was   explored by several authors.   This fact was used notably by Culler, Morgan and Shalen \cite{culler shalen, morgan shalen}  
 to construct compactifications of spaces of representations into $\PSL$ and obtain 
 new results on the geometry of 3-manifolds. In  rational dynamics  
 Kiwi \cite{kiwi puiseux, kiwi rational} used a similar construction to study the behaviour at infinity  
of families of cubic polynomials  or quadratic rational maps (see \cite{demarco mcmullen} 
for a different approach to this problem). 

  It would be natural to explore the interaction of bifurcation currents with   these compactifications, 
as well as the general bifurcation theory   of non-archimedian rational dynamical systems.

~

\end{document}